\documentclass[12pt]{amsart}%{article}%
\usepackage{amsfonts,latexsym}
\usepackage{amsmath}
\usepackage{amscd}
\usepackage[pdftex]{graphicx}
\usepackage[percent]{overpic}
\usepackage[pdftex]{color}

\newcommand{\nequation}{\setcounter{equation}{0}}

\newcommand{\R}{\mathbb{R}}

\newcommand{\llangle}{\langle\!\langle}
\newcommand{\rrangle}{\rangle\!\rangle}

\newcommand{\Rot}{\text{\upshape Rot}}
\newcommand{\Diff}{\text{\upshape Diff}}
\newcommand{\Met}{\text{\upshape Met}}
\newcommand{\Metmu}{\Met_{\mu}}
\newcommand{\Vol}{\text{\upshape Dens}}
\newcommand{\Diffmu}{\Diff_{\mu}}
\newcommand{\Isog}{\text{\upshape Iso}_g}
\newcommand{\DiffM}{\Diff(M)}
\newcommand{\MetM}{\Met(M)}
\newcommand{\MetmuM}{\Metmu(M)}
\newcommand{\VolM}{\Vol(M)}
\newcommand{\DiffmuM}{\Diffmu(M)}
\newcommand{\Diffmuex}{\Diff_{\mu,\text{\upshape ex}}}
\newcommand{\IsogM}{\Isog(M)}
\newcommand{\Volumizer}{\Xi}
\newcommand{\pullback}{\mathcal{P}}
\newcommand{\pullbackg}{{\pullback_g}}

\newcommand{\pullbackmu}{\pullback_{\mu}}

\newcommand{\imbedding}{\text{\upshape emb}}
\newcommand{\projection}{\text{\upshape proj}}
\newcommand{\transpose}{\dagger}

\newcommand{\Lie}{\mathcal{L}}
\newcommand{\ad}{\text{\upshape ad}}
\newcommand{\Ad}{\text{\upshape Ad}}
\newcommand{\grad}{\nabla}

\newcommand{\Laplacian}{\Delta}

\newcommand{\Isom}{\text{\upshape Isom}}

\newcommand{\Jac}{\text{\upshape Jac}}
\newcommand{\dist}{\text{\upshape dist}}
\newcommand{\diam}{\text{\upshape diam}}

\DeclareMathOperator{\diver}{div}
\DeclareMathOperator{\Tr}{Tr}

\DeclareMathOperator{\Ric}{Ric}

\DeclareMathOperator{\curl}{curl}

\newtheorem{theorem}{Theorem}[section]
\newtheorem{proposition}[theorem]{Proposition}

\newtheorem{remark}[theorem]{Remark}
\newtheorem{example}[theorem]{Example}
\newtheorem{figuretext}[theorem]{Figure}
\newtheorem{corollary}[theorem]{Corollary}

\input epsf
\title[Geometry of diffeomorphism groups] 
{Geometry of diffeomorphism groups, complete integrability and optimal transport} 
\author{B. Khesin}
\address{B.K.: Department of Mathematics, University of Toronto, M5S 2E4, Canada} 
\email{khesin@math.toronto.edu}

\author{J. Lenells}
\address{J.L.: Department of Mathematics, Baylor University, One Bear Place \#97328, Waco, TX 76798, USA.}
\email{Jonatan\_Lenells@baylor.edu}

\author{G. Misio\l ek}
\address{G.M.: Department of Mathematics, University of Notre Dame, IN 46556, USA}
\email{gmisiole@nd.edu}

\author{S. C. Preston}
\address{S.C.P.: Department of Mathematics, University of Colorado, CO 80309, USA}
\email{Stephen.Preston@colorado.edu}

\begin{document}
\maketitle

\today

\begin{abstract} 
\noindent 
We study the geometry of the space of densities $\VolM$, which is
the quotient space $\Diff(M)/\Diff_\mu(M)$ of the diffeomorphism group 
of a compact manifold $M$ by the subgroup of volume-preserving diffemorphisms, 
endowed with a right-invariant homogeneous Sobolev $\dot{H}^1$-metric. 
We construct an explicit isometry from this space to (a subset of) an infinite-dimensional sphere 
and show that the associated Euler-Arnold equation is a completely integrable system 
in any space dimension. We also prove that its smooth solutions break down in finite time. 

Furthermore, we show that the $\dot{H}^1$-metric induces the Fisher-Rao (information) metric 
on the space of probability distributions, and thus its Riemannian distance 
is the spherical version of Hellinger distance. 
We compare it to the Wasserstein distance in optimal transport which is induced by 
an $L^2$-metric on $\Diff(M)$. 
The $\dot{H}^1$ geometry we introduce in this paper can be seen as an infinite-dimensional version 
of the geometric theory of statistical manifolds. 
\end{abstract}



\tableofcontents

%%%%%%%%%%%%%%%%%%%%%%%%%%%%%%%%%%%%%%%%%%
%%%%%%%%%%%%%%%%%%%%%%%%%%%%%%%%%%%%%%%%%%%%%%
\section{Introduction} \nequation

The geometric approach to hydrodynamics pioneered by V. Arnold \cite{A} 
is based on the observation 
that the particles of a fluid moving in a compact $n$-dimensional Riemannian manifold $M$ trace out 
a geodesic curve 
in the infinite-dimensional group $\DiffmuM$ of volume-preserving diffeomorphisms 
(volumorphisms) of $M$. 
Arnold's framework is very general. 
It includes a variety of nonlinear partial differential equations of mathematical physics---in abstract form often referred to as 
\emph{Euler-Arnold} or \emph{Euler-Poincar\'e} equations. 

With a few exceptions, papers on infinite-dimensional Riemannian geometry, 
including diffeomorphism groups, 
tended to focus exclusively on either strong metrics 
or weak metrics of $L^2$-type.\footnote{Roughly speaking, 
a Riemannian metric on a Hilbert manifold is \textit{strong} if it generates 
a complete norm topology on each tangent space, i.e., if the Riemannian norm
is equivalent to the Hilbert norm; otherwise it is \textit{weak}.}
The interest in the latter has to do with the fact that 
such metrics often represent kinetic energies as in the case of hydrodynamics, 
see e.g., \cite{ak, Ebin-Marsden, shn}, 
or arise naturally in probability and optimal transport problems 
with quadratic cost functions (such as the Wasserstein or Kantorovich-Rubinstein distance);
see e.g., \cite{bb, otto, st, v}. 

On the other hand, in recent years there have appeared a number of 
interesting nonlinear evolution equations 
described as geodesic equations on diffeomorphism groups 
with respect to weak Riemannian metrics of Sobolev $H^1$-type, 
see e.g., \cite{ak, HMR, K-M} and their references for examples such as 
the Camassa-Holm, the Hunter-Saxton or the Euler-$\alpha$-equations.

In this paper we focus on the $H^1$ metrics both from a differential-geometric 
and a dynamical systems perspective. 
We will show that they arise naturally on (generic) orbits of diffeomorphism groups 
in the space of all Riemannian metrics on $M$. 
The main results of this paper concern 
the geometry of a subclass of such metrics, namely, degenerate right-invariant $\dot{H}^1$ Riemannian metrics 
on the full diffeomorphism group $\DiffM$ 
and the properties of solutions of the associated geodesic equations. 
The normalized metric is given at the identity diffeomorphism by 
\begin{equation} \label{diverdiv}
\llangle u,v\rrangle = \frac{1}{4} \int_M \diver{u}\cdot \diver{v} \, d\mu.
\end{equation} 
It descends to a non-degenerate Riemannian metric on 
the homogeneous space of right cosets (densities) $\VolM=\Diff(M)/\Diffmu(M)$. 
Furthermore, it turns out that the corresponding geometry is \textit{spherical} 
for any compact manifold $M$. 
More precisely, we prove that equipped with \eqref{diverdiv} 
the space $\VolM$ is isometric to (a subset of) an infinite-dimensional sphere 
in a Hilbert space. 
In fact, the metric we define on $\VolM$ can be viewed as an analog of 
the metric introduced by Otto \cite{otto} in the theory of mass transport; 
furthermore the Riemannian distance of \eqref{diverdiv}, which will be shown to coincide with the (spherical) Hellinger distance well-known in probability and mathematical statistics, 
can be viewed as an analogue of the $L^2$-Wasserstein distance. 
Remarkably, it also turns out that our metric induces 
the so-called Fisher-Rao (information) metric 
and related Chentsov-Amari $\alpha$-connections 
which have diverse applications in asymptotic statistics, information theory 
and quantum mechanics, see e.g., \cite{AN, chentsov}. 

We derive the Euler-Arnold equations associated to the general right-invariant $H^1$ metrics 
which include as special cases the $n$-dimensional (inviscid) Burgers equation, the Camassa-Holm equation,  
as well as variants of the Euler-$\alpha$ equation. 
In the particular case of the homogeneous $\dot{H}^1$-metric \eqref{diverdiv} 
the Euler-Arnold equation has the form 
\begin{align} \label{multiHS}
\rho_t + u \cdot \nabla \rho 
+ 
\tfrac{1}{2} \rho^2 
=
-\frac{\int_M \rho^2 \, d\mu}{2\mu(M)}  \,,
\end{align}
where $u = u(t,x)$ is a time-dependent vector field on $M$ with $\diver{u} = \rho$. 
This equation is a natural generalization of 
the completely integrable one-dimensional Hunter-Saxton equation \cite{H-S} 
which is also known to yield geodesics 
on the homogeneous space $\Diff(S^1)/\Rot(S^1)$ 
(the quotient of the diffeomorphism group of the circle by the subgroup of rotations), 
see \cite{K-M}. 

We  prove that the solutions of \eqref{multiHS} describe the great circles 
on a sphere in a Hilbert space and in particular the equation is 
a \textit{completely integrable} PDE for any number $n$ of space variables. 
The corresponding complete family of conserved integrals will be constructed 
in terms of angular momenta. 
We point out that with a few exceptions (such as 
the two-dimensional Kadomtsev-Petviashvili, Ishimori, and Davey-Stewartson equations) 
all known integrable evolution equations are limited to one space dimension.\footnote{As noted 
in \cite{F2006}, derivation and solution of integrable nonlinear evolution partial differential equations 
in three spatial dimensions has been the ``holy grail'' in the field of integrability since the late 1970s.} 
We hope that the geometric properties of (\ref{multiHS}) will make it 
an interesting novel example in this area. 

Furthermore, we show that the maximum existence time for smooth solutions 
of \eqref{multiHS} is necessarily finite for any initial conditions,
with the $L^\infty$ norm of the solution growing without bound as $t$ approaches the critical time. 
On the other hand, the geometry of the problem points to a method of constructing 
global weak solutions of \eqref{multiHS}. 
We will 
describe a strategy 
showing how this can be done using a technique of J. Moser. 

It is also of interest to consider the general form of 
the right-invariant ($a$-$b$-$c$) Sobolev $H^1$ metric on $\DiffM$ 
given at the identity by 
\begin{equation}\label{abcmetric}
\llangle u, v \rrangle 
:= 
a\int_M \langle u, v\rangle \, d\mu  
+  
b\int_M \langle \delta u^{\flat}, \delta v^{\flat} \rangle \, d\mu
+ 
c\int_M \langle du^\flat, dv^\flat\rangle \, d\mu,
\end{equation}
where $u, v \in T_e\Diff(M)$ are vector fields on $M$, 
$\mu$ is the Riemannian volume form,\footnote{The volume form $\mu$ 
is denoted by $d\mu$ whenever it appears under the integral sign.}  
$\flat$ is the isomorphism $TM \to T^*M$ 
defined by the metric on $M$
and $a, b$ and $c$ are non-negative real numbers.\footnote{We can allow $a<0$ 
as long as \eqref{abcmetric} is non-negative---e.g., when the first cohomology of $M$ is zero.} 
We derive the Euler-Arnold equation for the metric \eqref{abcmetric} below; while a detailed study of its geometry with the attendant curvature calculations 
will appear in a separate publication \cite{KLMP}. 

Finally, a comment on the functional analytic framework we chose for the paper. 
While our motivations and objectives are directly related to questions in analysis and PDE, 
in order to better present our geometric ideas we will---with few exceptions---work primarily with objects 
(function spaces, diffeomorphism groups, etc.) which consist of smooth functions. 
However, we emphasize that when these objects are equipped with a suitably strong topology 
(for example, any Sobolev $H^s$ topology with sufficiently large $s>n/2 +1$ 
will do for our purposes) 
then our constructions are rigorously justified in a routine manner. We will not belabour this point and instead refer the reader 
to the papers \cite{Ebin-Marsden} and \cite{MP} where such questions are considered 
in greater detail. 

\smallskip
The structure of the paper is as follows.
In Section \ref{background} we review the geometric background 
on Euler-Arnold equations and describe the space of densities 
used in optimal transport, as well as reductions of the $L^2$- and $H^1$-type metrics 
on $\DiffM$, its subgroup $\DiffmuM$, and their quotient. 

In Section \ref{aczero} we introduce the homogeneous $\dot H^1$-metric 
on the space of densities $\VolM = \Diff(M)/\Diffmu(M)$ 
and study its geometry. 
Generalizing the results of \cite{L1, L2} for the case of the circle 
we show that for any $n$-dimensional manifold 
the space $\VolM$ is isometric to 
a subset of the sphere in $L^2(M, d\mu)$ with the induced metric. 
The corresponding Riemannian distance is shown to be 
the spherical Hellinger distance. 

In Section \ref{statistics} we describe the relation of the $\dot{H}^1$-metric 
to geometric statistics and probability. 
In particular, we show that on the space $\VolM$ 
it plays the role of the classical Fisher-Rao metric. 
In the case $M=S^1$ we then use it to introduce the analogues of 
dual affine connections generalizing the constructions of Chentsov and Amari. 

In Section \ref{geodesics} we study local properties of solutions to the 
corresponding Euler-Arnold equation 
and demonstrate its complete integrability, as a geodesic flow on the sphere. 
Since for $M=S^1$ our equation reduces to the Hunter-Saxton equation 
we thus obtain an integrable generalization of the latter to any space dimension.

In Section \ref{global} we turn to global properties of solutions.
We derive an explicit formula for the Jacobian, 
prove that solutions necessarily break down in finite time 
and present an approach to construct global weak solutions. 

In Section \ref{abcmetricsection} we derive the Euler-Arnold equation
for the general $a$-$b$-$c$ metric \eqref{abcmetric} and show that 
several well-known PDE of mathematical physics can be obtained as special cases. 
We also discuss the situations in which some of the coefficients 
$a$, $b$, or $c$ are zero.

Finally, in Section \ref{sec:Met-orbits} we present a geometric construction 
which yields right-invariant metrics of the type \eqref{abcmetric} as induced metrics 
on the orbits of the diffeomorphism group from the canonical Riemannian $L^2$ structures 
on the spaces of Riemannian metrics and volume forms on 
the underlying manifold $M$. 

\medskip

{\bf Acknowledgements.}
We thank Aleksei Bolsinov, Nicola Gigli and Emanuel Milman 
for helpful suggestions.

%%%%%%%%%%%%%%%%%%%%%%%%%%%%%%%%%%%%%%%%%%%%%%%%%%%%%%%%%%%%%%%%%%%%%
%%%%%%%%
%%%%%%%%%%%%%%%%%%%%%%%%%%%%%%%%%%%%%%%%%%%%%%%%%%%%%%%%%%%%%%%%%%%%

\section{Geometric background}
\label{background}
\nequation 

%%%%%%%%%%%%%%%%%%%%%%%%%%%%%%%%%%%%%%%%%%%%%%%%%%%%%%%%%%%%%%%%%%%%%%%%%%
\subsection{The  Euler-Arnold equations} 
\label{subsec:EA} 

In this section we describe the general setup which is convenient to study geodesics 
on Lie groups and homogeneous spaces equipped with right-invariant metrics.

Let $G$ be a (possibly infinite-dimensional) Lie group 
with a group operation 
denoted by 
$(\eta, \xi)\mapsto \eta\circ \xi$. 
In our main examples group elements will be diffeomorphisms 
and the operation will be their composition. 
We shall use $T_eG$ to denote the Lie algebra of $G$, 
where $e$ is the identity element.
For any $\eta\in G$ the group adjoint is the map 
$\Ad_{\eta}\colon T_e G \to T_eG$ 
given by the differential 
$$
u \to \Ad_{\eta} u = (L_{\eta}\circ R_{\eta^{-1}})_{*e} u 
$$
where $L_{\eta}$ stands for a left-translation $\xi\mapsto \eta\,\circ\,\xi$ on the group 
while $R_\eta$ is the corresponding right-translation $\xi \mapsto \xi\,\circ\,\eta$. 
The algebra adjoint $\ad_u\colon T_eG\to T_eG$ is given by
$$
\ad_u = \frac{d}{dt}\Big|_{t=0} \Ad_{\eta(t)},
$$ 
where $\eta(t)$ is any curve in $G$ with 
$\eta(0)=e$ and $\dot{\eta}(0)=u$. 
If the group operation is composition of diffeomorphisms 
then in terms of the standard Lie bracket of vector fields 
we have $\ad_u v = -[u,v]$.

We equip $G$  with a right-invariant (possibly weak) Riemannian metric
which is determined by an inner product $\llangle\cdot , \cdot \rrangle$ 
on the tangent space at the identity 
$$
\llangle R_{\eta *e}u, R_{\eta *e}v\rrangle_{\eta} = \llangle u, v\rrangle\,,
$$
where $\eta\in G$ and $u,v\in T_eG$. 
The Euler-Arnold equation on the Lie algebra 
for the corresponding geodesic flow has the form 
\begin{equation}\label{utBuu}
u_t = -B(u,u) = -\ad_u^{*}u,
\end{equation}
with $u(0)=u_0$ and where the bilinear operator $B$ on $T_e G$ is defined by 
\begin{equation} \label{opB} 
\llangle B(u,v), w\rrangle 
= 
\llangle u, \ad_vw\rrangle. 
\end{equation} 
Equation \eqref{utBuu} describes the evolution in the Lie algebra of the vector $u(t)$ 
obtained by right-translating the velocity along the geodesic $\eta$ in $G$ 
starting at the identity with initial velocity $u(0)=u_0$.
The geodesic itself can be obtained by solving the Cauchy problem for the flow equation 
$$
\frac{d\eta}{dt} = R_{\eta \ast e} u, 
\quad 
\eta(0)= e. 
$$

%\begin{proof}
%This comes from minimizing the quantity 
%$$ \int_{t_1}^{t_2} \llangle \tfrac{d\eta}{dt}, \tfrac{d\eta}{dt} \rrangle \, dt 
%= 
%\int_{t_1}^{t_2} \llangle 
%\tfrac{d\eta}{dt} \circ \eta(t)^{-1}, \tfrac{d\eta}{dt} \circ\eta(t)^{-1} 
%\rrangle \, dt,$$ 
%with respect to a variation 
%$\frac{\partial \eta}{\partial s}\big|_{s=0} = y(t) \circ \eta$. 
%We have 
%$\frac{\partial}{\partial s}\big|_{s=0} \frac{\partial \eta}{\partial t} 
%= \frac{dy}{dt} + [u, y] = \frac{dy}{dt} - \ad_uy$, 
%where $u = \frac{d\eta}{dt} \circ \eta^{-1}$, 
%so that the criterion for minimizing is 
%$$ \int_{t_1}^{t_2} \llangle u, y_t + [u,y]\rrangle \, dt 
%= \int_{t_1}^{t_2} \llangle -u_t + B(u,u), y\rrangle \, dt = 0$$
%for every perturbation $y$. Hence we obtain \eqref{utBuu}.
%\end{proof}
%The corresponding adjoint and coadjoint actions on $G$ are given by
%$$\Ad_\eta u = T_{\eta^{-1}(\cdot)} \eta \cdot (u \circ \eta^{-1})$$
%and
%$$\Ad_\eta^* m = J(\eta) T^*\eta \cdot (m \circ \eta).$$ 

\begin{remark} \upshape
Rewriting Equation \eqref{utBuu} on the dual space  $T_e^*G$ in the form 
$$
\frac{d}{dt}( \Ad_\eta^\ast u ) =0, 
\quad
\llangle \Ad_\eta^\ast u, v \rrangle = \llangle u, \Ad_\eta v \rrangle 
$$
gives a conservation law
$
\Ad_{\eta(t)}^* u(t) = u_0 
$
expressing the fact that $u(t)$ is confined to one and the same coadjoint orbit 
during the evolution. 
\end{remark}

\begin{remark} \label{rem:descent} \upshape
Let $H$ be a closed subgroup of $G$. 
A right-invariant metric on $G$ descends to 
an invariant (under the right action of $G$) metric 
on the homogeneous space $G/H$ if and only if  
the projection of the metric to $T_e^\perp H\subset  T_eG$, 
the orthogonal complement to $H$ in the group $G$,
is bi-invariant with respect to the subgroup $H$ action.
%, i.e.  for all $u,v\in T_e^\perp H\subset T_eG$ and $w\in T_eH$ one has
%\begin{equation} \label{descent} 
%\llangle v, \ad_wu\rrangle + \llangle u, \ad_wv\rrangle = 0 \,.
%\end{equation} 
%
(If the metric in $G$ is degenerate along the subgroup $H$ then this
condition reduces to the metric bi-invariance  with respect to the $H$-action,
see e.g., \cite{K-M}.  We shall consider the general case in Section \ref{descends} below.)
The corresponding Euler-Arnold equation is then defined similarly 
as long as the metric is non-degenerate on the quotient $G/H$.
\end{remark}

%%%%%%%%%%%%%%
%%%%%%%%%%%%%%%%%%%%%%%%%%%%%%%%%%%

\subsection{Examples: equations of fluid mechanics}
\label{subsec:Ex} 
We list several equations of mathematical physics 
that arise as geodesic flows on diffeomorphism groups. 

%%%%%%%%%%%%%%
\subsubsection{}
Let $G=\Diffmu(M)$ be the group of volume-preserving 
diffeomorphisms (volumorphisms) of a closed Riemannian manifold $M$. 
Consider the right-invariant metric on $\DiffmuM$ generated by the $L^2$ inner product
\begin{equation} \label{rinvL2} 
\llangle u,v\rrangle_{L^2} 
= 
\int_{M} \langle u,v\rangle \, d\mu.
\end{equation} 
In this case the Euler-Arnold equation \eqref{utBuu} 
is the Euler equation of an ideal incompressible fluid in $M$ 
\begin{equation} \label{eq:Eulereq} 
u_t + \nabla_u u = -\grad p, \quad \diver{u}=0,
\end{equation} 
where $u$ is the velocity field and $p$ is the pressure function, see \cite{A}.
In the vorticity formulation the 3D Euler equation becomes 
$$
\omega_t + [u,\omega] = 0\,,\qquad \text{where} \quad 
\omega = \curl{u}\, .
$$

%%%%%%%%%%%%%
\subsubsection{} 
Consider the right-invariant metric on $\DiffmuM$ given by the $H^1$ inner product 
$$ 
\llangle u,v\rrangle_{H^1}  
= 
\int_M \big(\langle u,v\rangle 
+ 
\alpha^2 \langle du^{\flat}, dv^{\flat} \rangle \big) \, d\mu. 
$$
The corresponding Euler-Arnold equation is sometimes called the Euler--$\alpha$ 
(or Lagrangian-averaged) equation and in 3D has the form 
\begin{equation}\label{h1alpha}
\omega_t + [u,\omega] 
= 0 \,,
\qquad \text{where} \quad 
\omega = \curl{u} - \alpha^2 \curl\Delta{u};
\end{equation}
see e.g. \cite{HMR}. 

%%%%%%%%%%%%%%%%%%%%
\subsubsection{} 
Another source of examples is related to various
right-invariant Sobolev metrics on the group $G=\Diff(S^1)$ of all circle 
diffeomorphisms,  
as well as its one-dimensional central extension, the Virasoro group. 
Of particular interest are those metrics whose Euler-Arnold equations 
turn out to be completely integrable. 

On $\Diff(S^1)$ with the metric defined by the $L^2$ product 
the Euler-Arnold equation \eqref{utBuu} becomes 
the (rescaled) inviscid Burgers equation 
\begin{equation}  \label{invB1D}
u_t + 3 u u_x = 0 \,,
\end{equation} 
while the $H^1$ product yields the Camassa-Holm equation 
\begin{equation} \label{CH1D} 
  u_t - u_{txx} + 3uu_x - 2u_x u_{xx} - u u_{xxx} = 0\,. 
\end{equation} 
Similarly, the homogeneous part of the $H^1$ product gives rise to 
the Hunter-Saxton equation 
\begin{equation}\label{HS}
u_{txx} + 2 u_xu_{xx}  +uu_{xxx} = 0\,.
\end{equation}
More precisely, in the latter case one considers 
the quotient $\Diff(S^1)/\Rot(S^1)$ 
whose tangent space at the identity coset $[e]$, 
i.e., the coset corresponding to the identity diffeomorphism, 
can be identified with periodic functions of zero mean. 
The right-invariant metric on $\Diff(S^1)/\Rot(S^1)$ is defined by 
the $\dot H^1$ inner product on such functions 
$$
\llangle u,v\rrangle_{\dot{H}^1} = \int_{S^1} u_x v_x \,dx 
$$ 
and the corresponding Euler-Arnold equation is given by the
Hunter-Saxton equation \eqref{HS}. 

We also mention that if $G$ is the Virasoro group equipped with 
the right-invariant $L^2$ metric then the Euler-Arnold equation 
is the periodic Korteweg-de Vries equation 
$$
u_t + 3 u u_{x} + c u_{xxx} = 0\,,
$$
which is a shallow water approximation and the classical example of 
an infinite-dimensional integrable system. 
We refer the reader to \cite{K-M} for more details on these constructions.

\begin {remark}
{\rm
The Hunter-Saxton equation will be of particular interest to us in this paper. 
In \cite{L1, L2} Lenells constructed an explicit isometry between $\Diff(S^1)/\Rot(S^1)$ 
and a subset of the unit sphere in $L^2(S^1)$ 
and described the corresponding solutions of Equation (\ref{HS}) 
in terms of the geodesic flow on the sphere. 
Although the solutions exist classically only for a finite time they can be extended beyond 
the blowup time as weak solutions, see \cite{L3}. 
In the sections below, we shall show that this phenomenon can be established for flows 
on manifolds of arbitrary dimension. 
}
\end{remark}

%%%%%%%%%%%%%%%%
%%%%%%%%%%%%%%%%%%%%%%%%%%%%%

\subsection{The $L^2$-optimal transport and Otto's calculus} 
\label{sec:OT} 

Given a volume form $\mu$ on $M$ there is a natural fibration of the diffeomorphism group $\DiffM$ 
over the space of volume forms of fixed total volume $\mu(M)=1$. 
More precisely, the projection onto the quotient space $\DiffM/\DiffmuM$ 
defines a smooth ILH principal bundle\footnote{In the Sobolev category 
$\Diff^s(M) \to \Diff^s(M)/\Diff^s_\mu(M)$ 
is a $C^0$ principal bundle for any sufficiently large $s>n/2 +1$, see \cite{Ebin-Marsden}.} 
with fibre $\DiffmuM$ and whose base is diffeomorphic to 
the space $\Vol(M)$ of normalized smooth positive densities (or, volume forms) 
$$
\Vol(M)= \left\{ \nu \in \Omega^n(M) :~ \nu >0, \, \int_M d\nu  = 1 \right\}\,,
$$ 
see Moser \cite{moser}. 
Alternatively, let $\rho = d\nu/d\mu$ denote the Radon-Nikodym derivative 
of $\nu$ with respect to the reference volume form $\mu$. 
Then the base (as the space of constant-volume densities) can be regarded as  
a convex subset of the space of smooth functions on $M$ 
$$
\mathcal{M} = \left\{ \rho\in C^{\infty}(M, \mathbb{R}_{>0}) : \int_M \rho \, d\mu = 1 \right\}\,.
$$ 
In this case the projection map $\pi\colon \Diff(M)\to \mathcal{M}$ can be written explicitly as 
$\pi(\eta)=\Jac_\mu(\eta^{-1})$ 
where $\Jac_\mu(\eta)$ denotes the Jacobian of $\eta$ computed with respect to $\mu$, 
that is, $\eta^*\mu = \Jac_\mu(\eta)\mu$. 
%(Probabilists frequently prefer to use the map $\pi(\eta)=\Jac_\mu(\eta^{-1})$ 
%where $\Jac_\mu(\eta^{-1})$ is the density $\rho$ appearing in the change-of-variables formula 
%$\int_M f\circ \eta\, d\mu = \int_M \rho f \, d\mu$. 
%This only affects the sign in some formulas.)
%
%The subgroup $\DiffmuM$ projects to the (reference) density $\mu$ 
%and the fibre over any $\nu$ (with $\int_M d\nu = \int_M d\mu$) 
%consists of all diffeomorphisms that push $\mu$ to $\nu$.\footnote{Note that if $\eta$ 
%pushes $\mu$ to $\nu$ so do all diffeomorphisms of the form $\eta\circ\xi$ for $\xi\in \DiffmuM$ 
%constituting the right coset $[\eta]\in \DiffM/\DiffmuM$, see Figure \ref{??}.} 

The fact that $(\eta\circ \xi)^*\mu = \xi^*\eta^*\mu$ implies that 
\begin{equation}\label{jacobiancomposition}
\Jac_\mu(\eta\circ \xi)=
(\Jac_\mu(\eta)\circ \xi)\cdot \Jac_\mu(\xi).
\end{equation}
As a consequence, the projection $\pi$ satisfies $\pi(\eta\circ\xi) = \pi(\eta)$ whenever $\xi\in \Diffmu(M)$, i.e., 
whenever $\Jac_{\mu}(\xi)=1$. Thus $\pi$ is constant on the left cosets and descends to an isomorphism between
the quotient space of left cosets to the space of densities.

\begin{figure}
\begin{center}
 \includegraphics[width=.8\textwidth]{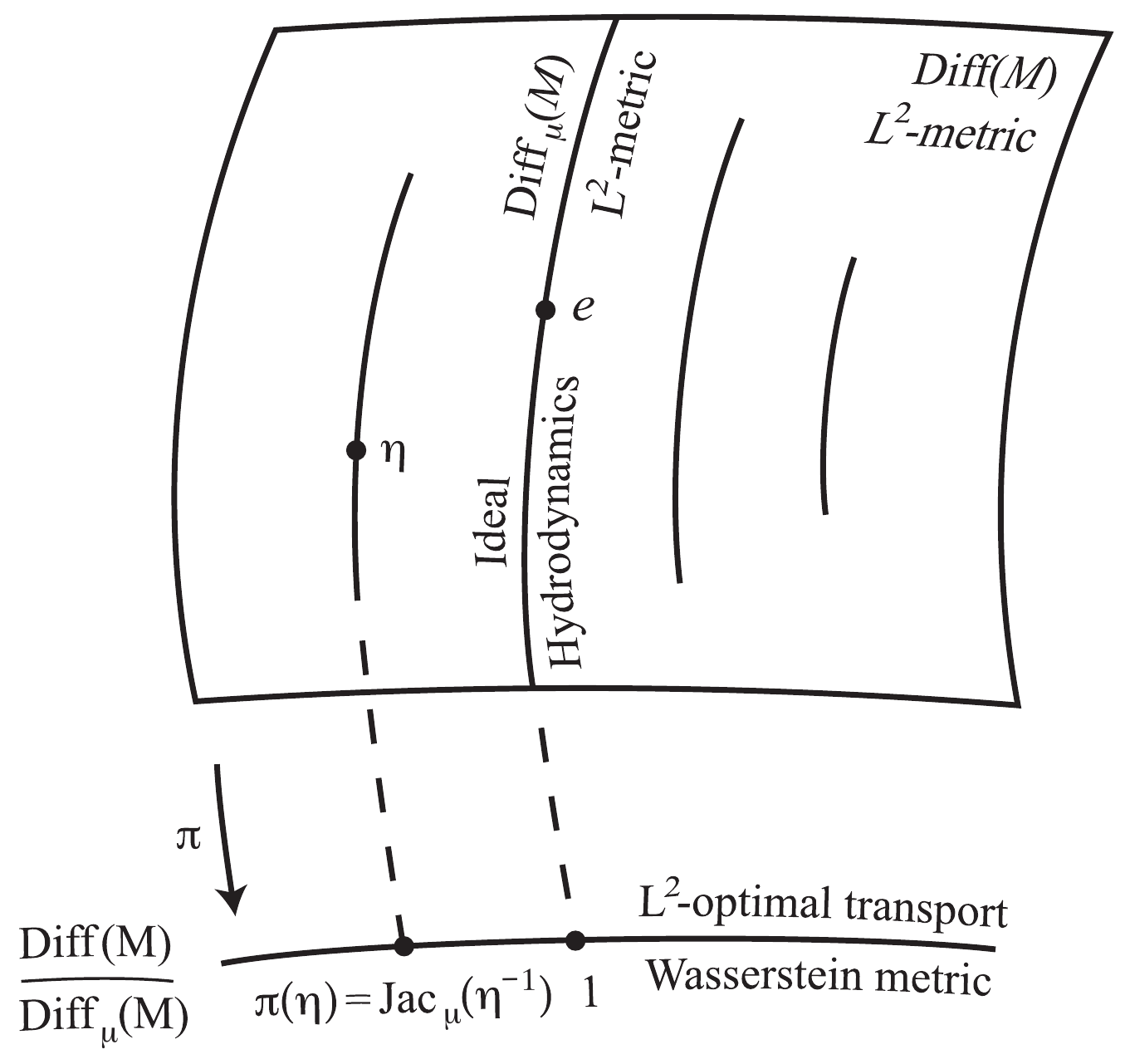}
      \begin{figuretext}\label{L2fibration.pdf}
        The fibration of $\DiffM$ with fiber $\DiffmuM$ determined by the reference density $\mu$ together with the $L^2$-metrics. The density $\mu$ itself corresponds to the constant function equal to 1.
                 \end{figuretext}
     \end{center}
\end{figure}

The group $\DiffM$ carries a natural $L^2$-metric 
\begin{equation}\label{dmetric}
\llangle u\circ\eta,v\circ\eta\rrangle_{L^2} 
= 
\int_M \langle u\circ\eta, v\circ\eta \rangle\, d\mu 
= 
\int_M \langle u, v \rangle \Jac_\mu(\eta^{-1}) \, d\mu
\end{equation}
where  $u,v \in T_e\DiffM$ and $\eta \in \DiffM$. This metric is neither left- nor right-invariant, 
although it becomes right-invariant when restricted to the subgroup $\DiffmuM$ of 
volumorphisms and becomes left-invariant only on the subgroup of isometries. 
Its significance comes from the fact that a curve $t \to \eta(t)$ in $\Diff(M)$ is a geodesic 
if and only if 
$t \to \eta(t)(p)$ is a geodesic in $M$ for each $p\in M$.\footnote{Note that if $M$ is {\it flat} 
then \eqref{dmetric} is a {\it flat} metric on $\Diff(M)$: 
a neighborhood of the identity is isometric to an open set in the pre-Hilbert space  
of vector-valued $H^s$ functions on $M$ equipped with the $L^2$ inner product.}  
Following Otto \cite{otto} one can then introduce a metric on the base 
for which the projection $\pi$ is a Riemannian submersion: 
vertical vectors at $T_{\eta}\Diff(M)$  are those fields $u\circ\eta$ with $\diver{(\rho u)}=0$, 
and horizontal fields are of the form $\grad f\circ\eta$ for some $f\colon M\to\mathbb{R}$, 
since the differential of the projection is $\pi_*(v\circ\eta) = -\diver{(\rho v)}$ 
where $\rho = \pi(\eta)$.\footnote{Otto's construction actually comes from $\Jac_\mu(\eta^{-1})$, 
which is important since it is left-invariant, not right-invariant.}

The induced metric on the base is then
\begin{equation}\label{gradmetric}
\llangle \alpha, \beta\rrangle_{\rho} = \int_M \rho \langle \grad f, \grad g\rangle \, d\mu,
\end{equation}
where $f$ and $g$ solve $\diver{(\rho \grad f)} = -\alpha$ and $\diver{(\rho \grad g)} = -\beta$ 
with mean-zero functions $\alpha$ and $\beta$ considered as elements of 
the tangent space at $\rho$. 

The geodesic equation of the metric \eqref{dmetric} on $\DiffM$ is 
$$
\dfrac{D}{dt} \dfrac{d\eta}{dt}=0
$$  
i.e., individual particles of the geodesic flow of diffeomorphisms $\eta(t)$ 
move along the geodesics in $M$ until they cross 
(and a smooth solution ceases to exist). 
In Eulerian coordinates, using 
$\frac{d}{dt}\eta(t,x) = u\big(t, \eta(t,x)\big)$, 
the geodesic equation can be rewritten 
as the pressureless Euler (or, inviscid Burgers) equation
\begin{equation}\label{pressurelesseuler}
\partial_t u + \nabla_u u=0
\end{equation}
and the induced geodesic equation on the quotient space reads 
\begin{equation}\label{gradientequation}
\partial_t \phi + \tfrac{1}{2} \lvert \grad \phi\rvert^2 = 0 
\end{equation} 
where $u = \grad \phi$ and $\phi : M \to \mathbb{R}$ 
is a smooth function.\footnote{The projection of the pressureless 
Euler equation onto the subgroup 
of volumorphisms $\DiffmuM$ is the Euler equation of incompressible fluids, 
so that in a sense fluid mechanics is ``orthogonal'' to optimal transport, as remarked in \cite{otto}.}  
Note that, as with any Riemannian submersion, if the tangent vector of a geodesic is initially horizontal, 
then it will remain so at later times. 
In our situation this corresponds to the fact that vorticity is conserved so that 
if a solution of \eqref{pressurelesseuler} is initially a gradient, it will always be a gradient 
(and $u=\grad \phi$ implies that $\phi$ satisfies \eqref{gradientequation}). 

The associated Riemannian distance in $\DiffM/\DiffmuM$ 
between two measures $\nu$ and $\lambda$ has an elegant interpretation 
as the $L^2$-cost of transporting one density to the other 
\begin{equation} \label{eq:wasser} 
\dist_W^2(\nu,\lambda)=\inf_\eta\int_M \dist_M^2(x,\eta(x))\, d\mu 
\end{equation} 
with the infimum taken over all diffeomorphisms $\eta$ such that $\eta^*\lambda = \nu$ 
%\footnote{A push-forward 
%of $\mu$ by a diffeomorphism $\eta$ is the same as its pull-back by $\eta^{-1}$.} 
and where $\dist_M$ denotes the Riemannian distance on $M$;
see \cite{bb} or \cite{otto}. 
The function $\dist_W$ is the $L^2$-Wasserstein (or Kantorovich-Rubinstein) distance 
between $\mu$ and $\nu$ and is of fundamental importance in optimal transport theory.

%%%%%%%%%%%%%%%%%%%%%%
%%%%%%%%%%%%%%%%%%%%%%%%%%

\subsection{Homogeneous metrics on $\DiffM/\DiffmuM$} 
%\subsection{$L^2$- and $H^1$-type metrics on $\Vol$} 
\label{descends}

In this section we formulate a condition under which metrics on $\DiffM$ 
descend to the homogeneous space of densities 
and describe several examples. 
The non-invariant $L^2$ metric used in optimal transport, 
as well as our main example, the right-invariant $\dot{H}^1$ metric \eqref{diverdiv}, 
both descend to the quotient. 
%(for the $\dot{H}^1$ metric this will be shown in Section \ref{abcmetricsection} below). 
But other natural candidates, such as the right-invariant $L^2$ metric on $\DiffM$ 
or the full $H^1$ metric, do not. 

We start with a general observation about right-invariant metrics. 
Let $H$ be a closed subgroup of a group $G$.

\begin{proposition}\label{KhesinMisiolekprop}
A right-invariant metric $\llangle \cdot, \cdot \rrangle$ on $G$ descends to 
a right-invariant metric on the homogeneous space $G/H$ 
if and only if  
the inner product restricted to $T_e^\perp H$ (the orthogonal complement of $T_eH$) 
is bi-invariant with respect to the action by the subgroup $H$, i.e.,  
for any $u,v\in T_e^\perp H\subset T_eG$ and any $w\in T_eH$ one has
\begin{equation} \label{descent} 
\llangle v, \ad_w u\rrangle + \llangle u, \ad_w v\rrangle = 0 \,.
\end{equation} 
\end{proposition}
\begin{proof}
The proof repeats with minor changes the proof for the case of 
a metric that is degenerate along a subgroup $H$; 
see \cite{K-M}. In the latter case condition \eqref{descent} reduces 
to bi-invariance with respect to the $H$-action 
and there is no need to confine to the orthogonal complement $T_e^\perp H$.
We only observe that in order to descend the orthogonal part of the metric 
must be $\Ad_H$-invariant 
$$
\llangle \Ad_h u, \Ad_h v\rrangle_{T^\perp H} 
= 
\llangle u, v\rrangle_{T^\perp H}
= 
\llangle u, v\rrangle 
\qquad\quad 
(h \in H)
$$
for any $u,v\in T_e^\perp H$. 
The corresponding condition for the Lie algebra action is obtained by differentiation. 
\end{proof}

\begin{remark} 
\upshape
It can be checked that the condition in Proposition \ref{KhesinMisiolekprop} is precisely what one needs in order for the projection map from $G$ to $G/H$ to be a Riemannian submersion, i.e., that the length of every horizontal vector is preserved under the projection. Since the metric on $G$ is assumed right-invariant, this condition reduces to one that can be checked  in the tangent space $T_e G$ at the identity.
\end{remark}

\begin{example} \upshape 
The degenerate right-invariant $\dot{H}^1$ metric \eqref{diverdiv} 
on $\DiffM$ descends to a non-degenerate metric on the quotient $\VolM=\DiffM/\DiffmuM$. 
The skew symmetry condition \eqref{descent} in this case will be verified 
in Section \ref{abcmetricsection} below (Corollary \ref{descentmetricvolum}). 
\end{example} 

\begin{example}\upshape  
The right-invariant $L^2$-metric \eqref{rinvL2} does not verify \eqref{descent} 
and hence does not descend to $\VolM$. 
In fact, if we set $u=v=\nabla{f}$ then for any vector field $w$ with $\diver{w}=0$ 
integration by parts gives
\begin{align*} 
\llangle \ad_w v, v \rrangle_{L^2}
= 
\int_M \langle \nabla_v w, v \rangle \, d\mu 
= 
\int_M \langle f \nabla\Delta{f}, w \rangle \, d\mu       
= 
- \int_M \langle (\Delta{f}) \nabla{f}, w \rangle \, d\mu \,,
\end{align*}
where we used the identity 
$\diver{\nabla_{\nabla f} w} = \diver{\nabla_w\nabla{f}} - w \cdot \Delta f$. 
It is not difficult to find $f$ and $w$ such that the above integral is non-zero. 
For example, we can take $w$ to be the divergence-free part of the field
$(\Delta f) \nabla f$ and  arrange for a suitable $f \in C^\infty(M)$ so that $w\neq 0$. 

Similarly, it follows that the full $H^1$ metric on $\DiffM$ obtained by right-translating 
the $a$-$b$-$c$ product \eqref{abcmetric} 
also fails to descend to a metric on $\VolM$. 
Note that the $c$-term in \eqref{abcmetric} does not contribute in this case. 
\end{example}

\begin{example} \upshape 
As already pointed out in Section \ref{sec:OT} the non-invariant $L^2$ metric \eqref{dmetric} 
descends to Otto's metric on the quotient space whose Riemannian distance 
is the $L^2$-Wasserstein distance on $\Vol(M)$. 
This metric is invariant under the action of $\DiffmuM$. 
\end{example} 

%%%%%%%%%%%%%%%%%%
\begin{remark} \upshape 
One can also consider non-invariant Sobolev $H^1$ metrics analogous to 
the non-invariant metric \eqref{dmetric} on $\DiffM$. 
If the manifold $M$ is flat then 
%thinking of diffeomorphism $\eta$ close to the identity as a vector function, 
(identifying a neighbourhood of the identity in the diffeomorphism group 
with a neighbourhood of zero in a vector space) 
the energy functional of such a metric evaluated on a curve $\eta(t)$ 
will have the form 
$$
\int_0^1\int_M \big( 
a \left\langle \partial_t \eta,\partial_t \eta \right\rangle
+
b \left\langle \delta\partial_t \eta,\delta\partial_t \eta \right\rangle 
+ 
c \left\langle d \partial_t \eta, d \partial_t \eta \right\rangle 
\big)\,d\mu dt \,.
$$
The first term will not be affected by a volume-preserving change of variables. 
However, the terms involving derivatives in the space variable ($\delta$ and $d$) 
will not be conserved in general.
This argument can be developed to show that among metrics of this type 
the non-invariant $L^2$ metric (corresponding to $b=c=0$) 
is the only one descending to the homogeneous space of densities $\VolM$.
\end{remark} 

%%%%%%%%%%%%%%%%%

%It is clear that the relation between 
%the \textit{right-invariant} degenerate $\dot{H}^1$-metric \eqref{multiHSmetric} on $\DiffM$ 
%and the spherical Hellinger distance 
%is analogous to that between 
%the \textit{non-invariant} $L^2$-metric \eqref{dmetric} on $\DiffM$ 
%and the Wasserstein distance. 
%In both cases, the latter is the Riemannian distance function of 
%the corresponding metric induced on the quotient space $\VolM$ of densities. 
%
%This is summarized in the following table. 

\begin{table}[htdp]
\caption{The geometric structures associated with $L^2$ and $\dot{H}^1$ optimal transport.}
\begin{center}
\begin{tabular}{|c|c|c|} 
$\DiffM$ & $\DiffmuM$ & $\VolM=\DiffM/\DiffmuM$ \\
\hline
\begin{tabular}{c} 
\small{ $L^2$-metric } \\ 
\small{ (non-invariant) } 
\end{tabular} 
& \begin{tabular}{c} 
\small{ $L^2$-right invariant metric } \\ 
\small{ (ideal hydrodynamics) } 
\end{tabular} 
& \begin{tabular}{c} 
\small{ Wasserstein distance } \\ 
\small{ ($L^2$-optimal transport) } 
\end{tabular} \\
\hline
\begin{tabular}{c} 
\small{ $\dot{H}^1$-metric } \\ \small{ (right-invariant) } 
\end{tabular} 
& \begin{tabular}{c} 
\small{ Degenerate } \\ \small{ (identically vanishing) }
%\dot{H}_{curl}^1$-metric \\ (truncated Euler-$\alpha$) 
\end{tabular} 
& \begin{tabular}{c} 
\small{ Spherical Hellinger distance } \\ 
\small{ ($\dot{H}^1$-optimal transport) } 
\end{tabular} \\
\hline
\end{tabular}
\end{center}
\label{default}
\end{table}%

%%%%%%%%%%%%%%%%%%%%%%%%%%%%%%%%%%%%%%%%%%%%%%%%%%%
%%%%%%%%%%%%%%%%%%%%%%%%%%%%%%%%%%%%%%%%%%%
%%%%%%%%%%%%%%%%%%%%%%%%%%%%%%%%%%%%%%%%%%%%%

\section{The $\dot{H}^1$-spherical geometry of the space of densities} 
\label{aczero}
\nequation

In this section we study the homogeneous space of densities $\Vol(M)$ 
on a closed $n$-dimensional Riemannian manifold $M$ 
equipped with the right-invariant metric induced by 
the $\dot{H}^1$ inner product \eqref{diverdiv}, that is 
\begin{equation} \label{multiHSmetric}
\llangle u\circ \eta, v\circ\eta\rrangle_{\dot{H}^1} 
= \frac 14
\int_M \diver{u} \cdot\diver{v} \, d\mu 
\end{equation}
for any $u, v \in T_e\DiffM$ and $\eta \in \DiffM$. 
It corresponds to the $a=c=0$ term 
in the general ($a$-$b$-$c$) Sobolev $H^1$ metric \eqref{abcmetric} 
of the Introduction in which, to simplify calculations, we set $b=1/4$. 
(We will return to the case of $b>0$ in Sections \ref{abcmetricsection} 
and \ref{sec:Met-orbits}.) 

The geometry of this metric on the space of densities turns out to be particularly remarkable. 
Indeed, we prove below that $\VolM$ endowed with the metric \eqref{multiHSmetric} 
is isometric to an open subset of a round sphere in the space of square-integrable functions 
on $M$.\footnote{This construction has an antecedent in the special case 
of the group of circle diffeomorphisms considered by Lenells \cite{L1, L2}.} 
In the next section, we will show that \eqref{multiHSmetric} corresponds to 
the Bhattacharyya coefficient (also called the affinity) in probability and statistics 
and that it gives rise to a spherical variant of the Hellinger distance. 
Thus the right-invariant $\dot{H}^1$-metric provides good alternative notions of 
``distance'' and ``shortest path'' for (smooth) probability measures on $M$ 
to the ones obtained from the $L^2$-Wasserstein constructions 
used in standard optimal transport problems.

%%%%%%%%%%%%%%%%%%%%%%%%%%%%%%%%%
\subsection{An infinite-dimensional sphere} 
We begin by constructing an isometry between the homogeneous space of densities 
$\VolM$ and an open subset of the sphere of radius $r$ 
$$
S^\infty_r = \left\{f \in L^2(M, d\mu): \int_M f^2\, d\mu =  r^2 \right\}
$$ 
in the Hilbert space $L^2(M, d\mu)$.
%Under this isometry geodesics of the metric \eqref{multiHSmetric} in $\DiffM/\DiffmuM$ 
%will be mapped to segments of the great circles on the infinite-dimensional sphere. 
%

As before, we let $\Jac_\mu(\eta)$ denote the Jacobian of $\eta$ with respect to 
the reference form $\mu$ and let $\mu(M)$ stand for the total volume of $M$.
%Note that the Jacobian is well defined on cosets $[\eta]\in \DiffM/\DiffmuM$.

\begin{theorem} \label{isometricHSsphere}
The map  $\Phi:\DiffM \to L^2(M, d\mu)$ given by 
$$
\Phi: \eta \mapsto f = \sqrt{\Jac_\mu\eta} 
$$
defines an isometry from the space of densities $\VolM=\DiffM/\DiffmuM$ 
equipped with the $\dot H^1$-metric \eqref{multiHSmetric} 
to an open subset of the sphere $S^\infty_r \subset L^2(M, d\mu)$  
of radius 
$$
r = \sqrt{\mu(M)} 
$$ 
with the standard $L^2$ metric. 

For $s > n/2 + 1$ the map $\Phi$ is a diffeomorphism between $\Diff^s(M)/\Diff_\mu^s(M)$ 
and the convex open subset of $S^\infty_r \cap H^{s-1}(M)$ which
consists of strictly positive functions on $M$.
\end{theorem}
\begin{proof} 
First, observe that the Jacobian of any orientation-preserving diffeomorphism 
is a strictly positive function. 
Next, using the change of variables formula, we find that 
$$
\int_M \Phi^2(\eta) \, d\mu 
= 
\int_M \Jac_\mu\eta \, d\mu =  \int_M \eta^*\, d\mu =  \int_{\eta(M)} d\mu 
= 
\mu(M) 
$$
which shows that $\Phi$ maps diffeomorphisms into $S^\infty_r$. 
Furthermore, observe that since for any $\xi \in \DiffmuM$ we have 
$$
\Jac_\mu(\xi \circ \eta)\mu = (\xi \circ \eta)^*\mu = \eta^*\mu = \Jac_\mu(\eta)\mu;
$$
it follows that $\Phi$ is well-defined as a map 
from $\DiffM/\DiffmuM$.

Next, suppose that for some diffeomorphisms $\eta_1$ and $\eta_2$ we have 
$\Jac_\mu(\eta_1) = \Jac_\mu(\eta_2)$. 
Then $(\eta_1 \circ \eta_2^{-1})^*\mu = \mu$ from which we deduce that $\Phi$ is injective.
Moreover, differentiating the formula $\Jac_\mu(\eta)\mu = \eta^\ast \mu$ 
with respect to $\eta$ and evaluating at $U \in T_\eta \DiffM$, we obtain 
$$
\Jac_{\eta^*\mu}(U)
= 
\diver(U \circ \eta^{-1}) \circ \eta \; \Jac_\mu\eta. 
$$
Therefore, letting $\pi:\DiffM \to \DiffM/\DiffmuM$ denote the bundle projection, 
we find that 
\begin{align*}
%\langle d_\eta(\Phi \circ \pi) U, d_\eta (\Phi \circ \pi) V \rangle_{L^2}
\llangle (\Phi\circ\pi)_{*\eta}(U), (\Phi\circ\pi)_{*\eta}(V) \rrangle_{L^2}
&= 
 \frac 14\int_M (\diver u \circ \eta)\cdot (\diver v \circ \eta) \, \Jac_\mu \eta  \,  d\mu
	\\
&= 
 \frac 14 \int_M \diver u \cdot \diver v \, d\mu %\\ &
= 
\llangle U, V \rrangle_{\dot{H}^1},
\end{align*}
for any elements $U = u \circ \eta$ and $V = v \circ \eta$ in $T_{\eta}\DiffM$ 
where $\eta \in \DiffM$. 
This shows that $\Phi$ is an isometry.

When $s>n/2 +1$ the above arguments extend routinely to the category of Hilbert manifolds 
modelled on Sobolev $H^s$ spaces. 
The fact that any positive function in $S^\infty_r \cap H^{s-1}(M)$ 
belongs to the image of the map $\Phi$ follows from Moser's lemma \cite{moser} 
whose generalization to the Sobolev setting can be found for example in \cite{Ebin-Marsden}. 
\end{proof} 

As an immediate consequence we obtain the following result. 

\begin{corollary} \label{cor:sphere} 
The space $\VolM=\DiffM/\DiffmuM$ equipped with the right-invariant metric 
\eqref{multiHSmetric} has strictly positive constant sectional curvature equal to 
$1/\mu(M)$.
\end{corollary}
\begin{proof}
As in finite dimensions, sectional curvature of the sphere $S^\infty_r$ 
equipped with the induced metric is constant and equal to $1/r^2$. 
The computation is straightforward using for example the Gauss-Codazzi equations.
\end{proof}

It is worth pointing out that the bigger the volume $\mu(M)$ of the manifold 
the bigger the radius of the sphere $S^\infty_r$ and therefore, by the above corollary, 
the smaller the curvature of the corresponding space of densities $\VolM$. 
Thus, in the case of a manifold $M$ of infinite volume 
one would expect the space of densities 
with the $\dot{H}^1$-metric \eqref{multiHSmetric} to be ``flat.''
Observe also that rescaling the metric  (\ref{multiHSmetric}) to
$$
%\llangle u\circ \eta, v\circ\eta\rrangle_{\dot{H}^1} 
%= 
b\int_M \diver{u} \cdot\diver{v} \, d\mu
$$
changes the radius of the sphere to $r=2\sqrt b \sqrt{\mu(M)}$.

%%%%%%%%%%%%%%%%%%%%%%%%%%%%%%%%%%

\subsection{The $\dot{H}^1$-distance and $\dot{H}^1$-diameter of $\DiffM/\DiffmuM$} 

The \textit{right invariant} metric \eqref{multiHSmetric} induces 
a Riemannian distance between  densities (measures) of fixed total volume on $M$ 
that is analogous to the Wasserstein distance \eqref{eq:wasser} induced by 
the \textit{non-invariant} $L^2$ metric used in the standard optimal transport. 
It turns out that the isometry $\Phi$ constructed in Theorem \ref{isometricHSsphere} 
makes the computations of distances in $\VolM$ with respect to \eqref{multiHSmetric} 
simpler than one would expect by comparison with the Wasserstein case. 

%%%%%%%%
\begin{figure} 
\begin{center}
 \includegraphics[width=.8\textwidth]{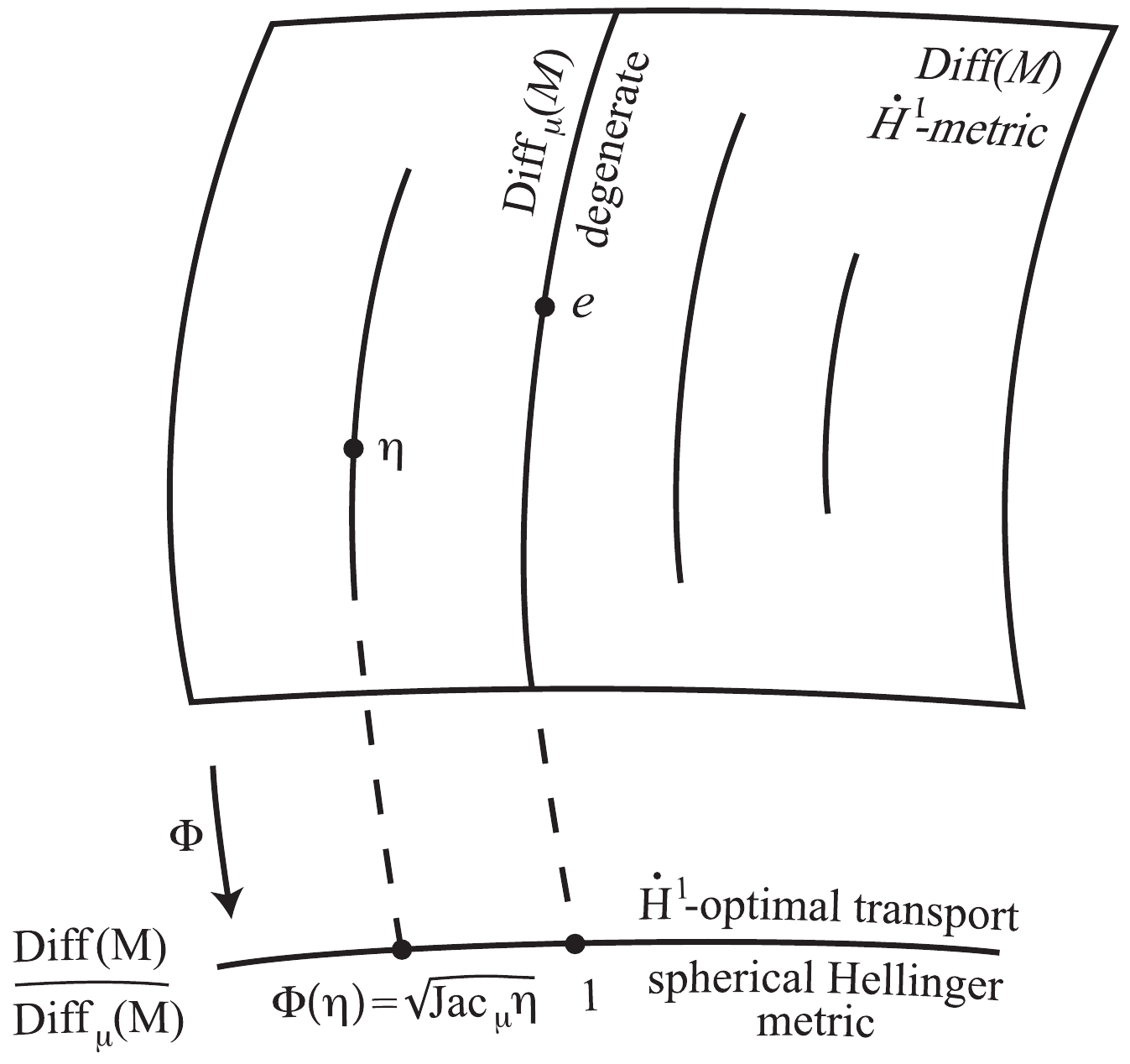}
     \begin{figuretext}\label{H1fibration.pdf}
        The fibration of $\DiffM$ with fiber $\DiffmuM$ determined by the reference density $\mu$ together with the $\dot{H}^1$-metric. 
         \end{figuretext}
     \end{center}
\end{figure}
%%%%%%%%

Consider two (smooth) measures $\lambda$ and $\nu$ on $M$ 
of the same total volume $\mu(M)$ which are absolutely continuous 
with respect to the reference measure $\mu$. 
Let $d\lambda/d\mu$ and $d\nu/d\mu$ 
be the corresponding Radon-Nikodym derivatives of $\lambda$ and $\nu$ 
with respect to $\mu$.

\begin{theorem} \label{distance}
The  Riemannian distance defined by the $\dot{H}^1$-metric \eqref{multiHSmetric} 
between measures $\lambda$ and $\nu$ in the density space $\VolM=\DiffM/\DiffmuM$ 
is 
\begin{equation} \label{dist_on_sphere}
\dist_{\dot H^1} (\lambda,\nu)
= 
\sqrt{\mu(M)}
\arccos{ \left( 
\frac{1}{\mu(M)} \int_M \sqrt{ \frac{d\lambda}{d\mu} \, \frac{d\nu}{d\mu} } \, d\mu 
\right) } \,.
\end{equation}

Equivalently, if $\eta$ and $\zeta$ are two diffeomorphisms mapping 
the volume form $\mu$ to $\lambda$ and $\nu$, respectively,
then the $\dot H^1$-distance between $\eta$ and $\zeta$ is
$$
\dist_{\dot H^1} (\eta,\zeta)
= 
\dist_{\dot H^1} (\lambda,\nu)
= 
\sqrt{\mu(M)}
\arccos \left( \frac{1}{\mu(M)}\int_M\sqrt{\Jac_\mu\eta \cdot \Jac_\mu\zeta}\,d\mu \right)\,.
$$ 
\end{theorem} 

\begin{proof} 
Let $f^2=d\lambda/d\mu$ and $g^2=d\nu/d\mu$. 
If $\lambda = \eta^\ast\mu$ and $\nu = \zeta^\ast \mu$ then using the explicit isometry $\Phi$ 
constructed in Theorem \ref{isometricHSsphere} it is sufficient to compute 
the distance between the functions 
$\Phi(\eta)=f$ and $\Phi(\zeta)=g$ 
considered as points on the sphere $S^\infty_r$ with the induced metric from $L^2(M,d\mu)$. 
Since geodesics of this metric are the great circles on $S^\infty_r$ 
it follows that the length of the corresponding arc joining $f$ and $g$ 
is given by 
$$
r \arccos{ \left(\frac{1}{r^2} \int_M f g\,d\mu \right) }\,,
$$
which is precisely formula \eqref{dist_on_sphere}.
\end{proof}

We can now compute precisely the diameter of the space of densities using standard formula 
$$
\diam_{\dot H^1} \,\Vol (M):=\sup{\big\{ \dist_{\dot H^1} (\lambda,\nu) :~\lambda, \nu \in \Vol(M) \big\} }.
$$

\begin{proposition} \label{diameter}
The diameter of the space $\Vol(M)$ equipped with the $\dot H^1$-metric 
\eqref{multiHSmetric} equals $\pi\sqrt{\mu(M)} /2$.
\end{proposition}
\begin{proof} 
First, observe that for any two densities $\lambda$ and $\nu$ in $\Vol(M)$ 
we have 
$$
\int_M\sqrt{ \frac{d\lambda}{d\mu} \, \frac{d\nu}{d\mu} } \, d\mu > 0. 
$$ 
The arc-cosine of this integral is less than $\pi/2$ which yields 
$\pi\sqrt{\mu(M)}/2$ 
as an upper bound for the diameter.\footnote{In particular, this means that the image 
of $\DiffM$ under $\Phi$ 
%which forms the cone of positive functions 
is a relatively small subset of $S^\infty_r$ 
whose points can be ``seen'' from the origin $0\in L^2(M, d\mu)$ 
at an angle less than $\pi/2$, see Figure \ref{HellingerDistance.pdf}.} 

In order to show that the diameter of $\Vol(M)$ is in fact equal to 
$\pi\sqrt{\mu(M)}/2$ 
we construct a sequence of measures such that the distance between them 
converges to this limit. 
Given any large $N$ consider a disk $D_N$ of volume $\mu(M)/N$ with respect to 
the reference measure $\mu$. 
Let $\nu_N$ be a smooth measure whose Radon-Nikodym derivative 
$f_N^2=d\nu_N/d\mu$ is a mollification of 
$N\cdot \chi_{D_N}$, 
where $\chi_{D_N}$ is the characteristic function of the disk, 
see Figure \ref{densities.pdf}. 
Note that the total volume of $\nu_N$ is the same as that of $\mu$. 

%%%%%%%%%%
\begin{figure}
\begin{center}
  \begin{overpic}[width=.6\textwidth]{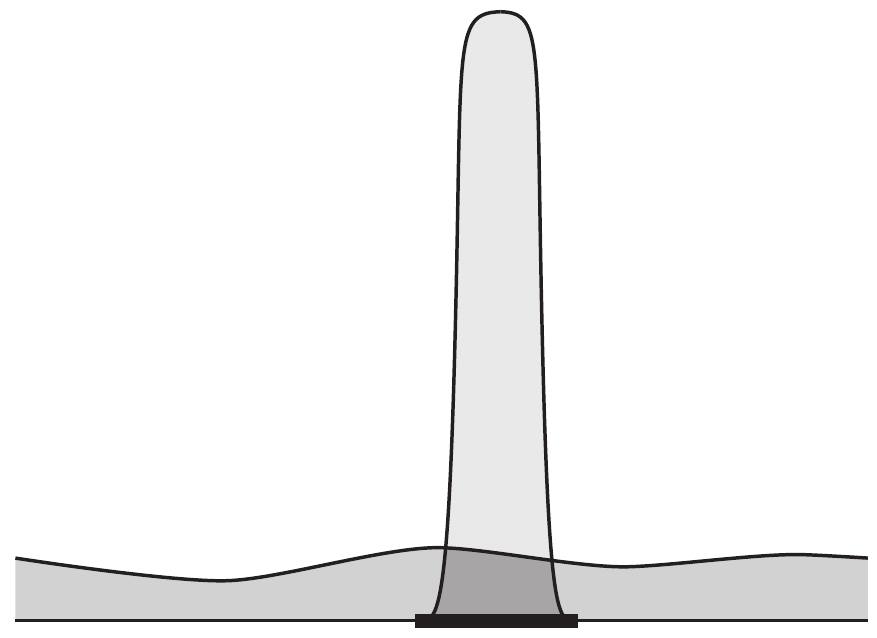}
    \put(53,-5){\large $D_N$}
    \put(44,65){\large $\nu_N$}
    \put(-4,8){\large $\mu$}
  \end{overpic}
  \bigskip
  \begin{figuretext}\label{densities.pdf}
        A density $\mu$  thinly distributed over the manifold and 
a peaked density $\nu_N$ with the same total volume as $\mu$.
  \end{figuretext}
  \end{center}
\end{figure}
%%%%%%%%%%

We will now estimate the distance between $\nu_N$ and $\mu$. 
Since on the disk the function $f_N$ is approximately equal to $\sqrt{N}$, 
the integral appearing in the argument of arc-cosine in \eqref{dist_on_sphere} 
can be estimated by 
$$
\int_M \sqrt{ \frac{d\mu}{d\mu} \, \frac{d\nu_N}{d\mu} } \, d\mu 
\simeq  
\int_{D_N}  \sqrt{f_N^2} \, d\mu 
\simeq 
\frac{\mu(M)}{N} \, \sqrt{N} \longrightarrow 0, 
\qquad \text{ as }\quad
N \to \infty. 
$$
It follows now that $\dist_{\dot H^1} (\mu,\nu_N) \to \pi\sqrt{\mu(M)} /2$ 
as $N \to \infty$ which completes the proof. 
\end{proof}

%%%%%%%%%%%%%

\begin{remark}[Applications to shape theory] \upshape 
It is tempting to apply the distance $\dist_{\dot{H}^1}$ to problems of computer vision 
and shape recognition. 

Given a bounded domain $E$ in the plane (a 2D ``shape'') 
one can mollify the corresponding characteristic function $\chi_E$ and 
associate with it (up to a choice of the mollifier) a smooth measure $\nu_E$ 
normalized to have total volume equal to 1. 
One can now use the above formula \eqref{dist_on_sphere} to introduce a notion of 
``distance'' between two 2D ``shapes'' $E$ and $F$ 
by integrating the product of the corresponding Radon-Nikodym derivatives 
%(mollified characteristic functions) 
with respect to the 2D Lebesgue measure. 
It is not difficult to check that 
$\dist_{\dot{H}^1}(\nu_E,\nu_F)= \pi/2$ 
when $\nu_E$ and $\nu_F$ are mutually singular 
and that 
$\dist_{\dot{H}^1}(\nu_E,\nu_F)=0$ 
whenever they coincide. 
This works also in the case when one of the measures is of delta-type and 
the other is very thinly distributed over a large area.

In this context it is interesting to compare the spherical metric 
to other right-invariant Sobolev metrics that have been introduced in shape theory. 
For example, in \cite{mm} the authors proposed to study 2D ``shapes'' using 
a certain K\"{a}hler metric on the Virasoro orbits of type $\Diff(S^1)/\Rot(S^1)$, 
see e.g., \cite{ky, tt}. 
This metric is particularly interesting because it is related to 
the unique complex structure on the Virasoro orbits.
Furthermore, it has negative sectional curvature. 
We refer to \cite{mm} for details. 
\end{remark}

%%%%%%%%%%%%%%%%%%%%%%%%%%%%%%%%%%%%%%%%%%%%%%%%%%%

\begin{example}[\textit{$\dot{H}^1$-gradients on the space of densities}] 
\label{ex:H1grad} 
\upshape
The $L^2$-Wasserstein metric 
\eqref{gradmetric} induced on the space of densities 
was used to study certain dissipative PDE (such as the heat and porous medium equations) as gradient flow equations on $\VolM$, see \cite{otto, v}. 
The analogous computations of the $\dot{H}^1$-gradients on $\VolM$ simplify due to 
the isometry with the round $L^2$-sphere discussed above.

For instance, let $H(\rho)$ be a functional of the general form
$$ 
H(\rho)=\int_M h(\rho)\,d\mu 
$$ 
where $\rho \in \VolM$ and $h$ is a smooth function on densities. 
%It will be convenient to represent $\rho$ by its 
%Radon-Nikodym derivative $\bar \rho=d\rho/d\mu$ and then the function is
%$h\in C^\infty(\mathbb{R})$ : $ H(\rho):=\int_M h(\bar\rho)\,d\mu$.
% \;\; \mathrm{and}\;\; h \in C^\infty(\mathbb{R}). 
Then one computes that $\mathrm{grad}_{\dot{H}^1} H = h'(\rho)$.

Here is a quick way to see this. 
For a small real parameter $\epsilon$ and any mean-zero function $\beta$ on $M$ 
we have 
$$ 
H(\rho+\epsilon\beta)
= 
\int_M h(\rho+\epsilon \beta)\,d\mu
= 
\int_M h(\rho) \, d\mu + \epsilon\int_M h'(\rho) \beta \, d\mu + \mathcal{O}(\epsilon^2)\,.
$$ 
By Theorem \ref{isometricHSsphere} one can perform the calculation using the $L^2$ metric 
on the sphere $S^\infty_r$ 
and identify the variational derivative $\delta H/\delta\rho$ of $H$ 
with its gradient $\mathrm{grad}_{\dot{H}^1} H$ so that 
$$
\Big\langle \frac{\delta H}{\delta \rho}, \beta \Big\rangle 
= 
\frac{d}{d\epsilon}H( \rho + \epsilon\beta ) \big|_{\epsilon=0} 
=
\int_M h'(\rho) \beta \, d\mu\, 
$$
which gives the result. 

Similarly, for the functional 
$$
F(\rho):=\frac 12\int |\nabla\rho|^2 \mu
$$
one obtains $\text{grad}_{\dot H^1} F(\rho)=-\Delta\rho$.
Observe that in this case the associated gradient flow equation 
$$
\partial_t \rho = -\text{grad}_{\dot H^1}F
$$
can be interpreted as the heat equation on densities 
$$
\partial_t \rho=\Delta \rho. 
$$ 
Thus the Dirichlet functional in the right-invariant $\dot H^1$ metric on $\VolM$ 
yields the same heat equation as 
the Boltzmann (relative) entropy functional 
$E(\rho) = \int_M \rho \log{\rho} \, d\mu$ 
in the $L^2$-Wasserstein metric. 
This provides yet one more relation of these two gradient approaches to
the heat equation, discussed in \cite{ags}. 
\end{example} 

\begin{remark}  \label{rem:H1grad} \upshape 
Alternatively, one can carry out the calculation directly in the diffeomorphism group $\DiffM$ 
as follows. 
Let $t \to \eta(t)$ be a curve in $\DiffM$ with $\eta(0)=\eta$ and $\dot{\eta}(0)=v\circ\eta$ 
where $v = \nabla{f} \in T_e\DiffM$ 
and let $\rho = \Phi(\eta)$ where $\Phi: \DiffM \to \mathrm{Dens}(M)$ 
is the projection, see Figure \ref{H1fibration.pdf}. 
Consider a functional $H(\rho)$ lifted to $\DiffM$
as an invariant functional. 
We seek the $\dot{H}^1$-gradient of the functional $H(\eta)$ 
at $\eta\in \DiffM$ 
in the form $\nabla{g} \circ \eta$ for some function $g:M \to \mathbb{R}$. 
One has
\begin{align*} 
\llangle \mathrm{grad}_{\dot{H}^1}{H}, v\circ\eta \rrangle_{\dot{H}^1} 
&= 
\frac{d}{d t}\bigg\vert_{t=0} 
\int_M h\big( \Jac_\mu^{1/2}\eta(t)  \big) d\mu     \\ 
&= 
\frac{1}{2} \int_M 
h' \big( \Jac_\mu^{1/2}\eta \big) \Jac_\mu^{-1/2}\eta \, \mathrm{div}\, v\circ\eta \, \Jac_\mu\eta \, d\mu \,, 
\end{align*}
where 
\begin{equation} \label{eq:DTJ} 
\tfrac{\partial}{\partial t}\big\vert_{t=0} \Jac_\mu \eta(t) 
= 
(\mathrm{div}\, v \circ\eta) \cdot \Jac_\mu\eta. 
\end{equation} 
After changing variables in the integrand, using \eqref{multiHSmetric}, 
and projecting with the help of $\Phi$ we obtain
$$
\Phi_{\ast\eta} (\mathrm{grad}_{\dot{H}^1}{H}) 
= 
\frac{1}{2} \mathrm{div}\, (\mathrm{grad}_{\dot{H}^1}{H}\circ\eta^{-1})\circ\eta \cdot \Jac_\mu^{1/2}\eta 
= 
h'(\rho) 
$$ 
which coincides with the formula obtained above directly on $\VolM$.
\end{remark}

%%%%%%%%%%%%%%%%%%%%%%%%%%%%%%%%%%%%%%%%%%%%%%%%%%%
%%%%%%%%%%%%%%%%%%%%%%%%%%%%%%%%%%%%%%%%%%%%%%%%%%%
%%%%%%%%%%%%%%%%%%%%%%%%%%%%%%%%%%%%%%%%%%%%%%%%%%%

\section{Probability and infinite-dimensional geometric statistics}
\label{statistics}
\nequation

\subsection{Spherical Hellinger distance} 

The Riemannian distance function $\dist_{\dot{H}^1}$ on the space of densities $\VolM$ 
introduced in Theorem \ref{distance} is very closely related to 
the Hellinger distance in probability and statistics. 

Recall that given two probability measures $\lambda$ and $\nu$ on $M$ 
that are absolutely continuous with respect to a reference probability $\mu$ 
the {\it Hellinger distance} between $\lambda$ and $\nu$ is defined as 
$$
\dist^2_{Hel}(\lambda, \nu) 
= 
\int_M \Bigg( \sqrt{\frac{d\lambda}{d\mu}} 
- 
\sqrt{\frac{d\nu}{d\mu}} ~\Bigg)^2 d\mu \,. 
$$
As in the case of $\dist_{\dot{H}^1}$ one checks that 
$\dist_{Hel}(\lambda, \nu)=\sqrt{2}$ 
when $\lambda$ and $\nu$ are mutually singular
and that 
$\dist_H(\lambda, \nu)=0$ 
when the two measures coincide. 
It can also be expressed by the formula 
$$
\dist_{Hel}(\lambda, \nu) = \sqrt{ 2\big( 1 - BC(\lambda, \nu) \big) }
$$
where 
$BC(\lambda, \nu)$ is the so-called Bhattacharyya coefficient (affinity) 
used to measure the  ``overlap'' between statistical samples,
see e.g., \cite{chentsov} for more details. 

In order to compare the Hellinger distance $\dist_{Hel}$ with 
the Riemannian distance $\dist_{\dot{H}^1}$ defined in \eqref{dist_on_sphere} 
recall that probability 
measures $\lambda$ and $\nu$ are normalized
by the condition $\lambda(M)=\nu(M)=\mu(M)=1$.
As before, we shall consider the square roots of 
the respective Radon-Nikodym derivatives as points on the (unit) sphere in $L^2(M, d\mu)$. 
One can immediately verify the following two corollaries of Theorem \ref{isometricHSsphere}. 

\begin{corollary}
The Hellinger distance $\dist_{Hel}(\lambda, \nu)$ 
between the normalized densities $d\lambda=f^2 d\mu$ and $d\nu=g^2 d\mu$
is equal to the distance in $L^2(M, d\mu)$ 
between the points on the unit sphere $f,g \in S^\infty_1 \subset L^2(M,d\mu)$. 
\end{corollary}

\begin{corollary}
The Bhattacharyya coefficient $BC(\lambda, \nu)$ 
for two normalized densities 
$d\lambda=f^2 d\mu$ and $d\nu = g^2 d\lambda$ 
is equal to the inner product of the corresponding positive
functions $f$ and $g$ in $L^2(M, d\mu)$ 
$$
BC(\lambda, \nu) 
= 
\int_M \sqrt{ \frac{d\lambda}{d\mu} \, \frac{d\nu}{d\mu}} \, d\mu 
= 
\int_M f g \, d\mu. 
$$
\end{corollary}
 
Let $0 < \alpha < \pi/2$ denote the angle between $f$ and $g$ viewed as unit vectors 
in $L^2(M, d\mu)$. 
Then we have 
$$ 
\dist_{Hel}(\lambda, \nu) = 2 \sin(\alpha/2)
\quad 
\text{and} 
\quad 
BC(\lambda, \nu) = \cos{\alpha}\,, 
$$ 
while 
$$ 
\dist_{\dot{H}^1}(\lambda, \nu) = \alpha = \arccos{BC(\lambda, \nu)}. 
$$

\begin{figure}
\begin{center}
\bigskip
\bigskip
 \begin{overpic}[width=.6\textwidth]{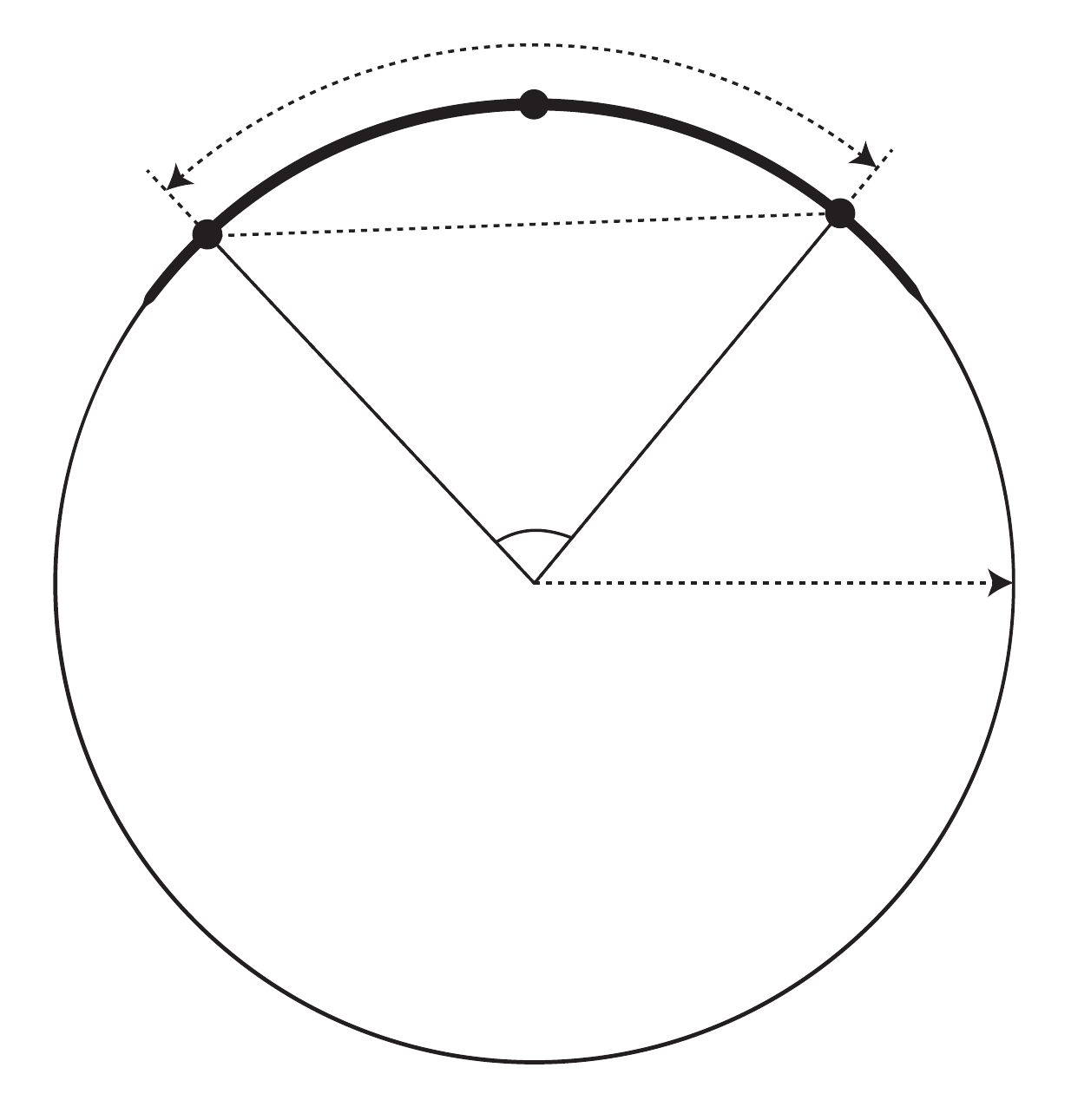}
      \put(46,54){\large $\alpha$}
      \put(45.5,43.5){\large $0$}
      \put(-8,78){\large $f = \Phi(\lambda)$}
      \put(79,80){\large $g = \Phi(\nu)$}
      \put(38,84){\large $\Phi(e) = 1$}
      \put(34,74){\large $\dist_{Hel}(\lambda, \nu)$}
      \put(36, 99){\large $\dist_{\dot{H}^1}(\lambda, \nu)$}
      \put(65, 50){\large $r = 1$}
      \put(81, 14){\large $S_1^\infty \subset L^2(M, d\mu)$}
    \end{overpic}
     \begin{figuretext}\label{HellingerDistance.pdf}
        The Hellinger distance $\dist_{Hel}(\lambda, \nu)$ and the spherical Hellinger distance $\dist_{\dot{H}^1}(\lambda, \nu)$ between two points $f = \Phi(\lambda)$ and $g = \Phi(\nu)$ in $S^\infty_1$. The thick arc represents the image of $\DiffM$ under the map $\Phi$.
     \end{figuretext}
     \end{center}
\end{figure}

%DEFINITION 
Thus, we can refer to the Riemannian distance $\dist_{\dot{H}^1}(\lambda, \nu)$ 
on $\Vol(M)$ as the \textit{spherical Hellinger distance} 
between $\lambda$ and $\nu$.

%%%%%%%%%%%%%%%%%%%%%%%%%%%%%%%%%%%%%%%%%%%%%%%%%%%

\subsection{Fisher-Rao information in infinite-dimensions} 
\label{sub:H1GS}

It is remarkable that 
the right-invariant $\dot{H}^1$ metric introduced in formulas 
\eqref{diverdiv} and \eqref{multiHSmetric} 
provides a convenient framework 
for an \textit{infinite-dimensional} Riemannian geometric approach to mathematical statistics. 
Efforts directed toward finding suitable differential geometric approaches to statistics 
have a long history going back to the work of 
Fisher, Rao \cite{rao} and Kolmogorov 
and continued with 
Chentsov \cite{chentsov}, Efron \cite{efron} and Amari-Nagaoka \cite{AN}. 

In the classical approach one considers 
finite-dimensional families of probability distributions on $M$ whose elements 
are parameterized by subsets $E$ of the Euclidean space $\mathbb{R}^k$,
$$
\mathcal{S} 
= 
\left\{ 
\nu=\nu_{s_1, \dots , s_k} \in \mathcal{M}:  (s_1, \dots, s_k) \in E \subset \mathbb{R}^k 
\right\}\, .
$$
When equipped with a structure of a smooth $k$-dimensional manifold such a family 
is referred to as a statistical model. 
Rao \cite{rao} showed that any $\mathcal{S}$ carries 
a natural structure given by a $k\times k$ positive definite matrix 
\begin{equation} \label{fisherrao} 
I_{ij} 
= 
\int_M \frac{\partial \log{\nu}}{\partial s_i} \frac{\partial \log{\nu}}{\partial s_j} \nu \, d\mu 
\qquad 
(i, j = 1, \dots , k) \,,
\end{equation} 
called the Fisher-Rao (information) metric.\footnote{The significance of this metric for statistics 
was also pointed out by Chentsov~\cite{chentsov}.} 
%\footnote{An interesting example of Rao 
%shows that the two-parameter family of Gaussian distributions is isometric to the hyperbolic plane 
%of constant negative curvature.} \footnote{\bf I am confused: so the curvature is not always positive?}

\smallskip 
In our approach we shall regard a statistical model $\mathcal{S}$ 
as a \textit{$k$-dimensional Riemannian submanifold} of 
the \textit{infinite-dimensional Riemannian manifold} of probability densities $\VolM$ 
defined on the underlying $n$-dimensional compact manifold $M$. 
The following theorem shows that the Fisher-Rao metric \eqref{fisherrao}  
is (up to a constant multiple) the metric induced on the submanifold $\mathcal{S}\subset \VolM$ 
by the (degenerate) right-invariant Sobolev $\dot H^1$-metric \eqref{diverdiv} 
we introduced originally on the full diffeomorphism group $\DiffM$. 

\begin{theorem} \label{GStat} 
The right-invariant Sobolev $\dot{H}^1$-metric \eqref{multiHSmetric} 
on the quotient space $\Vol(M)$ of probability densities on $M$ 
coincides with the Fisher-Rao metric on any $k$-dimensional statistical submanifold of $\VolM$. 
\end{theorem} 
\begin{proof} 
We carry out the calculations directly in $\DiffM$. 
Given any $v$ and $w$ in $T_e\DiffM$ consider a two-parameter family of diffeomorphisms 
$(s_1, s_2) \to \eta(s_1,s_2)$ in $\DiffM$ 
starting from the identity $\eta(0,0)=e$ with 
$\frac{\partial}{\partial s_1}\eta(0,0)=v$, $\tfrac{\partial}{\partial s_2}\eta(0,0)=w$ 
and let 
$$
v(s_1,s_2) \circ \eta(s_1,s_2) = \tfrac{\partial}{\partial s_1} \eta(s_1,s_2) 
\quad\mathrm{and}\quad  
w(s_1,s_2) \circ \eta(s_1,s_2) = \tfrac{\partial}{\partial s_2} \eta(s_1,s_2)
$$ 
be the corresponding variation vector fields along $\eta(t,s)$. 

If $\rho$ is the Jacobian of $\eta(s_1,s_2)$ computed with respect to 
the fixed measure $\mu$ then \eqref{fisherrao} takes the form 
$$
I_{vw} 
= 
\int_M
\frac{\partial}{\partial s_1} \Big( \log{\Jac_\mu\eta(s_1,s_2)} \Big) 
\frac{\partial}{\partial s_2} \Big( \log{\Jac_\mu\eta(s_1,s_2)} \Big) 
\Jac_\mu \eta(s_1,s_2) \, d\mu. 
$$
Recall from \eqref{eq:DTJ} that 
$$
\frac{\partial}{\partial s_1} \Jac_\mu\eta(s_1,s_2) 
= 
\mathrm{div}\, v(s_1,s_2) \circ \eta(s_1,s_2) \cdot \Jac_\mu \eta(s_1,s_2)
$$
and similarly for the partial derivative in $s_2$. 
Using these and changing variables in the integral we now find 
\begin{align*} 
I_{vw} 
&=
\int_M 
\frac{\tfrac{\partial}{\partial s_1}\Jac_\mu\eta(s_1,s_2) \tfrac{\partial}{\partial s_2}\Jac_\mu\eta(s_1,s_2)}
{\Jac_\mu\eta(s_1,s_2)} \big|_{s_1=s_2=0} 
\, d\mu       \\ 
&=
\int_M
\big( \mathrm{div}\, v \circ \eta \big) 
\cdot 
\big( \mathrm{div} \, w \circ \eta \big) \, \Jac_\mu \eta 
\, d\mu   \\ 
&= 
\int_{M} \mathrm{div}\, v  \cdot \mathrm{div}\, w \, d\mu    
=
4 \llangle v, w \rrangle_{\dot{H}^1}\,,
\end{align*} 
from which the theorem follows. 
\end{proof} 

\begin{remark} \upshape 
Theorem \ref{GStat} suggests that the $\dot{H}^1$ counterpart of optimal transport 
with its associated spherical Hellinger distance 
is the infinite-dimensional version of geometric statistics sought in \cite{AN} and \cite{chentsov}. 
%
%For example, 
%it is possible to express the Cram\'er-Rao inequality for the variance-covariance matrix of 
%an unbiased estimator using the $\dot{H}^1$-gradient of Example \ref{ex:H1grad}. 
\end{remark}

%%%%%%%%%%%%%%%%%%%%%%%%%%%%%%%%%%%%%%%%%%%%%%

\subsection{Affine connections and duality on $\Diff(S^1)/\Rot(S^1)$} 
\label{sub:affine} 

One of the questions posed by Amari and Nagaoka asked for an infinite-dimensional theory 
of so-called dual connections, see e.g. \cite{AN}, Section 8.4. 
Dual connections arise naturally in classical affine geometry and 
have turned out to be particularly useful in geometric statistics. 

In this section we describe a family of such connections $\nabla^{(\alpha)}$ 
on the density space $\VolM$ in the case when $M=S^1$ 
which generalize the $\alpha$-connections of Chentsov \cite{chentsov}. 
We will show that $\nabla^{(\alpha)}$ and $\nabla^{(-\alpha)}$ are dual 
with respect to the right-invariant $\dot{H}^1$-metric \eqref{multiHSmetric} 
and derive the (reduced) geodesic equations for $\nabla^{(\alpha)}$. 
When $\alpha=0$ we recover, as expected, 
the one-dimensional Hunter-Saxton equation \eqref{HS}. 
However, in the case $\alpha=-1$ we obtain another 
completely integrable system called the $\mu$-Burgers equation, see \cite{lmt}. 
The general case of an arbitrary compact manifold $M$ will be described in \cite{KLMP}. 

We refer to \cite{AN} for basic facts about dual connections 
on finite-dimensional statistical models 
and to \cite{K-M} and \cite{L2} for details on the $\dot{H}^1$ geometry of 
the homogeneous space of densities $\Diff(S^1)/\Rot(S^1)$. 
It will be convenient to identify the latter with the set of circle diffeomorphisms 
which fix a prescribed point 
$$
\Vol(S^1) \simeq \left\{ \eta \in \Diff(S^1): \eta(0)=0 \right\}. 
$$ 
Denote $A=-\partial_x^2$ and given a smooth mean-zero periodic function $u$ 
define the operator $A^{-1}$ by 
$$
A^{-1} u(x) = - \int_0^x \int_0^y u(z) \, dz dy + x\int_0^1\int_0^y u(z) \, dz dy.
$$

Let $v$ and $w$ be smooth mean-zero functions on the circle and denote by 
$V =v\circ\eta$ and $W =w\circ\eta$ the corresponding vector fields on $\Vol(S^1)$. 
For any $\alpha \in \mathbb{R}$ we define 
\begin{equation} \label{alfa-con} 
\eta \to ( \nabla^{(\alpha)}_VW ) (\eta) 
= 
\left( w_x v + \Gamma^{(\alpha)}_e(v,w) \right) \circ\eta 
\end{equation} 
where 
\begin{equation} \label{alfa-ch} 
\Gamma^{(\alpha)}_e (v,w) = \frac{1+\alpha}{2} A^{-1} \partial_x (v_x w_x). 
\end{equation}

\begin{proposition} \label{prop:alfa} 
For each $\alpha \in \mathbb{R}$ the map $\nabla^{(\alpha)}$ 
is a right-invariant torsion-free affine connection on $\Vol(S^1)$ 
with Christoffel symbols $\Gamma^{(\alpha)}$. 
$\nabla^{(0)}$ is the Levi-Civita connection of the $\dot{H}^1$-metric \eqref{multiHSmetric}, 
while $\nabla^{(-1)}$ is flat. 
\end{proposition} 
\begin{proof} 
The first and last assertions are verified in a routine manner using 
\eqref{alfa-con} and \eqref{alfa-ch}. 
That $\Gamma^{(0)}$ is precisely the Christoffel symbol of 
the unique (weak) Riemannian $\dot{H}^1$ connection can be found e.g. in \cite{L2}. 
\end{proof} 

Following \cite{AN} we say that two connections $\nabla$ and $\nabla^\ast$ 
on $\Vol(S^1)$ are dual with respect to $\llangle \cdot, \cdot \rrangle$ if 
$$
U \llangle V, W \rrangle 
= 
\llangle \nabla_UV, W \rrangle + \llangle V, \nabla^\ast_UW \rrangle 
$$
for any smooth vector fields $U$, $V$ and $W$. 

\begin{proposition} \label{prop:duality} 
The affine connections $\nabla^{(\alpha)}$ and $\nabla^{(-\alpha)}$ on $\Vol(S^1)$ 
are dual with respect to the $\dot{H}^1$-metric \eqref{multiHSmetric} 
for any $\alpha \in \mathbb{R}$. 
\end{proposition} 
\begin{proof} 
It suffices to work with right-invariant vector fields 
$U=u\circ\eta$, $V=v\circ\eta$ and $W=w\circ\eta$ 
where $u,v$ and $w$ are periodic mean-zero functions on the circle. 
By right-invariance we have $U \llangle V, W \rrangle_{\dot{H}^1} = 0$. 
On the other hand, from \eqref{alfa-con} and \eqref{multiHSmetric} 
for any $\alpha$ we obtain 
\begin{align*} 
&\llangle \nabla^{(\alpha)}_UV, W \rrangle_{\dot{H}^1} 
+ 
\llangle V, \nabla^{(-\alpha)}_UW \rrangle_{\dot{H}^1} 
=   \\ 
%\llangle (\nabla^{(\alpha)}_UV)\circ\eta^{-1}, W\circ\eta^{-1} \rrangle_{\dot{H}^1} 
%+ 
%\llangle V\circ\eta^{-1}, (\nabla^{(\alpha)}_UW)\circ\eta^{-1} \rrangle_{\dot{H}^1}      \\ 
&= 
\frac{1}{4} \int_0^1 
\big( v_x u + \tfrac{1+\alpha}{2} A^{-1} (u_x v_x)_x \big)_x w_x \, dx 
+ 
\frac{1}{4} \int_0^1 
v_x \big( w_x u + \tfrac{1-\alpha}{2} A^{-1} (u_x w_x)_x \big)_x \, dx  \\ 
&= 
\frac{1}{4} \int_0^1 \Big( 
(v_xu)_x w_x - \tfrac{1+\alpha}{2} u_x v_x w_x + v_x (w_x u)_x - \tfrac{1-\alpha}{2} v_x u_x w_x 
\Big) dx 
= 0 
\end{align*} 
using integration by parts and the explicit formula for $A^{-1}$ 
The general case is reduced to the above in a routine way as in e.g. \cite{Ebin-Marsden}. 
\end{proof} 

%\begin{remark} \upshape 
From the formula \eqref{alfa-ch} we see that the Christoffel symbols $\Gamma^{(\alpha)}$ 
do not lose derivatives. In fact, with a little extra work it can be shown that this implies 
that $\nabla^{(\alpha)}$ is a smooth connection on the $H^s$ Sobolev completion of 
$\Vol(S^1)$ to a Sobolev-Hilbert manifold if $s>3/2$. 
Consequently, one establishes the existence and uniqueness in $H^s$ of 
local (in time) geodesics of $\nabla^{(\alpha)}$ using the methods of \cite{MP} or \cite{lmt}. 
%\end{remark} 

We next derive the geodesic equations of $\nabla^{(\alpha)}$ 
(reduced to the tangent space at the identity) 
but will not pursue well-posedness questions here. 

\begin{proposition} \label{EulerArnold-alfa} 
The equation of geodesics of the affine $\alpha$-connection is $\nabla^{(\alpha)}$ is 
\begin{equation} \label{eq:alfa} 
u_{txx} + (2-\alpha) u_x u_{xx} + u u_{xxx} = 0.
\end{equation} 
The cases $\alpha=0$ and $\alpha=-1$ correspond to one-dimensional 
completely integrable systems: 
the HS equation \eqref{HS} and the $\mu$-Burgers equation, respectively. 
\end{proposition} 
\begin{proof} 
The equation for geodesics of $\nabla^{(\alpha)}$ on $\Vol(S^1)$ reads 
$$
\ddot{\eta} + \Gamma^{(\alpha)}_\eta (\dot\eta, \dot\eta ) = 0
$$
where $\Gamma^{(\alpha)}_\eta$ is the right-translation of $\Gamma^{(\alpha)}_e$. 
Substituting $\dot{\eta} = u\circ\eta$ gives 
$$
u_t + u u_x + \Gamma^{(\alpha)}_e(u,u) = 0 
$$
and using \eqref{alfa-ch} and differentiating both sides of the equation twice 
in the $x$ variable completes the proof. 
\end{proof} 

\begin{remark} \upshape 
Dual connections of Amari have not yet been fully explored in infinite dimensions. 
We add here that as in finite dimensions \cite{AN} there is a simple relation between 
the curvature tensors of $\nabla^{(\alpha)}$ i.e.  
$ R^{(\alpha)} = (1 - \alpha^2)R^{(0)}$ 
where $R^{(0)}$ is the curvature of the round metric on $\Vol(S^1)$.
It follows that the dual connections $\nabla^{(-1)}$ and $\nabla^{(1)}$ are flat 
and in particular there is a chart on $\Vol(S^1)$ in which the geodesics of the latter 
are straight lines. 
This permits to write down smooth solutions of 
the Cauchy problem for \eqref{eq:alfa} in the case $\alpha =1$, i.e. 
$$
u_{txx} + u_x u_{xx} + u u_{xxx} = 0,   
\quad 
u(0,x) = u_0(x) 
$$ 
(with $u_0(0) = 0$) in the explicit form 
$$
u(t,x) 
= 
\left( \int_0^1 e^{tu_{0x}(y)}  dy \right)^{-1} 
\int_0^{\eta_t^{-1}(x)} u_{0x}(y) e^{tu_{0x}(y)} dy  
$$
where 
$\eta_t(x) \equiv \eta(t,x) = \int_0^x e^{tu_{0x}(y)} dy / \int_0^1 e^{tu_{0x}(y)} dy$ 
is the flow of the solution $u$. 
\end{remark}

%%%%%%%%%%%%%%%%%%%%%%%%%%%%%%%%%%%%%%%%%%%%%%%%%%%
%%%%%%%%%%%%%%%%%%%%%%%%%%%%%%%%%%%%%%%%
%%%%%%%%%%%%%%%%%%%%%%%%%%%%%%%%%%%%%%%%%%%

\section{The geodesic equation: solutions and integrability} 
\label{geodesics}
\nequation

In the preceding sections we studied the geometry of 
the $\dot{H}^1$-metric \eqref{multiHSmetric} on the space of densities $\VolM$. 
In this section we shall focus on obtaining explicit formulas for solutions of 
the Cauchy problem for the associated Euler-Arnold equation 
and prove that they necessarily break down in finite time. 
Furthermore, we will show how to view the Euler-Arnold equation 
as a \textit{completely integrable} Hamiltonian system in any space dimension.

%%%%%%%%%%%%%%%%%%%%%%%%%%%%%%

\subsection{Classical solutions of the $\dot H^1$ Euler-Arnold equation} 
The formalism of Section \ref{subsec:EA} applied to the $\dot{H}^1$-metric 
on the group $\DiffM$ yields the following equation. 

\begin{proposition} \label{generalizedhuntersaxtonprop}
The Euler-Arnold equation of the right-invariant $\dot{H}^1$-metric 
\eqref{multiHSmetric} on the homogeneous space $\VolM$ reads 
\begin{equation}\label{generalizedhuntersaxton} 
\rho_t + u \cdot \nabla \rho 
+ 
\tfrac{1}{2} \rho^2 
=
-\frac{\int_M \rho^2 \, d\mu}{2\mu(M)}  \,,
\qquad\text{ where}\qquad %   \\  \label{ghs-ic} 
\rho=\diver u\,.
\end{equation}
\end{proposition}
Note that in the special case $M=S^1$ differentiating 
Equation \eqref{generalizedhuntersaxton} with respect to the space variable 
gives the Hunter-Saxton equation \eqref{HS}. 
We postpone the proof of Proposition \ref{generalizedhuntersaxtonprop} 
until Section \ref{abcmetricsection} where we derive the Euler-Arnold equation 
in the general case of the right-invariant Sobolev $H^1$-metric \eqref{abcmetric} on $\DiffM$, 
see Corollary \ref{descentmetricvolum}. 

\begin{remark} \upshape 
The right-hand side of Equation \eqref{generalizedhuntersaxton} 
is independent of time for any initial condition $\rho_0$ because the integral 
$\int_M\rho^2\,d\mu$ corresponds to the energy (the squared length 
of the velocity) in the $\dot H^1$-metric on $\VolM$ 
and is constant along a geodesic. 
This invariance will also be verified by a direct computation in the proof below.
\end{remark}

Consider an initial condition in the form  
\begin{align} \label{HSic} 
%u(0,x)= u_0(x) 
%\qquad 
%\text{or (equivalently)} 
%\qquad 
\rho(0,x) = \diver{u}_0(x). 
\end{align} 
We already have an indirect method for solving 
the initial value problem for Equation \eqref{generalizedhuntersaxton} 
by means of Theorem \ref{isometricHSsphere}. 
We now proceed to give explicit formulas for the corresponding solutions. 
\begin{theorem}\label{huntersaxtonsoln}
Let $\rho=\rho(t,x)$ be the solution of the Cauchy problem \eqref{generalizedhuntersaxton}-\eqref{HSic} 
and suppose that $t \mapsto \eta(t)$ is the flow of the velocity field $u=u(t,x)$, i.e., 
$\tfrac{\partial}{\partial t}\eta(t,x) = u(t,\eta(t,x))$ where $\eta(0,x)=x$. 
Then 
\begin{equation}\label{explicitgenhssoln}
\rho\big(t, \eta(t,x)\big) 
= 
2\kappa\tan{\left( \arctan{ \frac{\diver{u_0(x)}}{2\kappa} } -\kappa t\right)},
\end{equation}
where 
\begin{equation}\label{kappadef}
\kappa^2 = \frac{1}{4\mu(M)} \int_M (\diver{u_0})^2 \, d\mu. 
\end{equation}
Furthermore, the Jacobian of the flow is 
\begin{equation}\label{lagrangianhuntersaxtonsoln}
\Jac_\mu\big(\eta(t,x)\big) 
= 
\Big( 
\cos{\kappa t} + \frac{\diver{u_0}(x)}{2\kappa} \sin{\kappa t}
\Big)^2.
\end{equation}
\end{theorem}
\begin{proof}
For any smooth real-valued function $f(t,x)$ the chain rule gives 
$$ 
\frac{d}{d t} \big( f(t,\eta(t,x)) \big) 
= 
\frac{\partial f}{\partial t}(t, \eta(t,x)) 
+ 
\left\langle u\big(t,\eta(t,x)\big), \nabla{f}\big(t,\eta(t,x)\big) \right\rangle \,.
$$ 
Using this we obtain from \eqref{generalizedhuntersaxton}  an equation 
for $f=\rho\circ\eta$ 
\begin{equation}\label{ODE}
\frac{df}{dt} 
+
\tfrac{1}{2} f^2 
= 
-C(t) \,,
\end{equation}
where 
$C(t) = (2\mu(M))^{-1}\int_M \rho^2 d\mu$, as remarked above, is in fact independent of time. 
Indeed, direct verification gives 
\begin{align*}
\mu(M) \frac{dC(t)}{dt} 
&= 
\int_M \rho \rho_t \, d\mu 
=
\int_M \diver{u} \diver{u_t} \, d\mu     \\
&= 
- \int_M \langle u, \nabla\diver{u}\rangle \diver{u} \, d\mu 
- 
\frac{1}{2} \int_M (\diver{u})^3 \, d\mu =0 \,,
\end{align*}
where the last cancellation follows from integration by parts. 

Set $C=2\kappa^2$. Then, for a fixed $x \in M$ the solution of 
the resulting ODE in \eqref{ODE} with initial condition $f(0)$ 
has the form 
$$ 
f(t) 
= 
2\kappa \tan\big( \arctan{ ( f(0)/2\kappa )} - \kappa t \big),
$$
which is precisely \eqref{explicitgenhssoln}.

In order to find an explicit formula for the Jacobian we first compute the time derivative 
of $\Jac_\mu(\eta)\mu$ to obtain 
$$
\frac{d}{dt} \big( \Jac_\mu(\eta) \mu \big) 
= 
\frac{d}{dt} ( \eta^\ast \mu ) 
= 
\eta^\ast( \Lie_u \mu ) 
= 
\eta^\ast( \diver{u}\, \mu) 
= 
(\rho\circ\eta) \, \Jac_\mu(\eta) \mu\,. 
$$
This gives a differential equation for $\Jac_\mu\eta$, which we can now solve 
with the help of \eqref{explicitgenhssoln} 
to get the solution in the form of \eqref{lagrangianhuntersaxtonsoln}.
\end{proof}

Note that \eqref{lagrangianhuntersaxtonsoln} completely determines 
the Jacobian 
regardless of any ``ambiguity'' in the velocity field $u$ satisfying 
$\diver u=\rho$ in equation \eqref{generalizedhuntersaxton}.
The reason is that  the Jacobians can be considered as elements of 
the quotient space $\VolM=\DiffM/\DiffmuM$.
(It will be convenient later to resolve the ambiguity by choosing velocity 
as the gradient field $u= \nabla\Delta^{-1}\rho$.)

\begin{remark}[Great circles on $S^\infty_r$]\label{greatcircles} \upshape 
We emphasize that formula (\ref{lagrangianhuntersaxtonsoln}) 
for the Jacobian $\Jac_\mu\eta$ of the flow 
is best understood in light of the correspondence between 
geodesics in $\VolM$ and those on the infinite-dimensional 
sphere $S^\infty_r$ established in Theorem \ref{isometricHSsphere}. 
Indeed, the map
$$ 
t \to \sqrt{\Jac_\mu\big(\eta(t,x)\big)} 
= 
\cos{\kappa t} + \frac{\diver{u_0}(x)}{2\kappa} \sin{\kappa t} 
$$
describes the great circle 
on the sphere $S^\infty_r \subset L^2(M, d\mu)$ 
passing through the point $1$ with initial velocity 
$\Phi_{\ast e} (u_0) = \frac 12 \diver{u_0}$. 
\end{remark}

Having an explicit formula for the Jacobian $\Jac_\mu\eta$ 
raises a question whether it is possible to construct an associated global flow 
$\eta : \mathbb{R}\times M \to M$ of the velocity field $u$ 
which could be interpreted as a global (in time) weak solution 
of equation \eqref{generalizedhuntersaxton}. 
We address this question in Section \ref{Heuristics}.

%%%%%%%%%%%%%%%%%%%%%%%%%%%%%%

\subsection{Integrability of the geodesic flow on a sphere} 
\label{complete_int} 
\nequation 

Next, we examine in more detail the Hamiltonian structure of 
the $\dot{H}^1$-geodesic flow on $\VolM$. 
First, however, we recall the finite dimensional case. 

%%%%%%%%%%%%
\subsubsection*{The geodesic flow on the standard sphere}
Consider the unit sphere $S^{n-1}\subset\R^n$ given by the equation 
$\sum_{j=1}^n q_j^2 = 1$ with $q= (q_1, \dots, q_n) \in \mathbb{R}^n$ 
and equipped with its standard round metric. 
The geodesic flow in this metric is defined by the Hamiltonian 
$H=\sum_{j=1}^n p_j^2$ 
on the cotangent bundle $T^*S^{n-1}$.
It is a classical example of a completely integrable system, 
which has the property that all of its orbits are closed. 
Indeed, the projections of these orbits from $T^*S^{n-1}$ to $S^{n-1}$ 
are the great circles.

\begin{proposition}
The functions $h_{ij}=p_iq_j-p_jq_i, ~1\le i<j\le n$ on $T^*\R^n$ 
(as well as their reductions to $T^*S^{n-1}$) 
commute with the Hamiltonian $H=\sum_{j=1}^n p_j^2$ 
and generate the Lie algebra $\mathfrak{so}(n)$.
\end{proposition}
\begin{proof}
The Hamiltonian functions $h_{ij}$ in $T^\ast \R^n$ generate rotations 
in the $(q_i,q_j)$-plane in $\R^n$, which are isometries of $S^{n-1}$.
These rotations commute with the geodesic flow on the sphere 
and hence $\{h_{ij},H\}=0$.

A direct computation gives  $\{h_{ij},h_{jk}\}=h_{ik}$, 
which are the commutation relations of $\mathfrak{so}(n)$.
\end{proof}

This example illustrates the so-called noncommutative integrability: 
the geodesic flow may possess many first integrals 
(in the case of the sphere there are $n(n-1)/2$ of them),  
however they need not be in involution.  
Nevertheless, one can proceed as follows, see \cite{Bolsinov} for  details. 

Consider the functions 
\begin{align*} 
H_1&:=h^2_{12},   \\ 
H_2&:=h^2_{12}+h^2_{13}+h^2_{23},   \\ 
 &\; \; \; \vdots 
  \\ 
 H_{n-1} & :=h^2_{12}+\dots+h^2_{n-1\,n}
=
\sum_{j=1}^n p_j^2 \sum_{j=1}^n q_j^2 
-
\Big( \sum_{j=1}^n q_jp_j \Big)^2. 
\end{align*} 
On the cotangent bundle $T^*S^{n-1}$ the function $H_{n-1}$ coincides with 
the Hamiltonian $H$ 
since $\sum_{j=1}^n q^2_j=1$ 
and 
$\sum_{j=1}^n p_iq_i=0$ 
(``the tangent plane equation''). 
We thus have 

\begin{proposition}
The functions $\left\{H_i ~:~i=1,...,n-1\right\}$ form a complete set of independent integrals 
in involution for the geodesic flow on the round sphere $S^{n-1}\subset\R^n$, 
that is $\{H_i,H_j\}=0$, for any $1 \leq i, j \leq n-1$. 
\end{proposition}
\begin{proof}
A routine calculation. 
\end{proof} 

Alternatively, one can consider the chain of subalgebras
$\mathfrak{so}(2)\subset \mathfrak{so}(3) \subset...\subset \mathfrak{so}(n)$. 
Then $H_k$ is one of the Casimir functions for $\mathfrak{so}(k+1)$ 
and it therefore commutes with any function on $\mathfrak{so}(k+1)^*$. 
In particular, it commutes with all the preceding functions $H_m$ for $m<k$. 
They are functionally independent because at each step $H_k$ 
involves new functions $h_{jk}$.

\begin{remark} \upshape 
Note that all the Hamiltonians are quadratic in momenta.
In principle, it should be possible to compute them as limits 
of the quadratic first integrals associated with the geodesic flow on an ellipsoid, 
as the ellipsoid semi-axes $a_i$ approach the value $1$ with different orders 
$1+\epsilon^i$ for $i=1, \dots, n$;
we again refer to \cite{Bolsinov} for more details. 
\end{remark} 

\begin{remark} \label{rem:biham}\upshape 
The geodesic flow on an $n$-dimensional elipsoid is known to be
a bi-hamiltonian  dynamical system and its first integrals can be obtained by a procedure similar to the Lenard-Magri scheme, see e.g., \cite{Tabachnikov, Matveev-Topalov}. 
On the other hand, observe that the limit case of the geodesic flow on the sphere 
in a sense 
manifests the ``highest degree'' of integrability: 
the corresponding Arnold-Liouville tori are $1$-dimensional and all orbits are closed.
\end{remark}

%%%%%%%%%%%%%%%%%%%%%%%%%%%%%%%%%%%%
\subsubsection*{The geodesic flow on an infinite-dimensional sphere}
The finite-dimensional construction of the first integrals $H_k$ described above 
allows the following infinite-dimensional generalization.

Consider the unit sphere $S^\infty_1 \subset L^2(M, d\mu)$. 
In fact, choose an orthonormal basis  $\{ e_j \}_{j=1}^{\infty}$ in $L^2(M, d\mu)$ 
(for example, the normalized  eigenfunctions of the Laplace-Beltrami 
operator on $M$)  and write 
$$
S^\infty_1
%=
%\left\{ q \in L^2(M, d\mu): ~\int_M q^2 \, d\mu = 1 \right\}
=
\left\{ \sum_{j=1}^\infty q_je_j : ~\sum_{j=1}^\infty |q_j|^2 = 1 \right\}.
$$

The Hamiltonian of the geodesic flow on the sphere $S^\infty_1$ 
is given by the integral 
$H=\int_M p^2\,d\mu$ 
where $p\in L^2(M, d\mu)$.

As in the previous subsection we define functionals
$h_{ij} =p_iq_j-p_jq_i, ~1\le i<j$, 
the infinitesimal generators of rotations in the $(q_i,q_j)$-plane.
They now form the Lie algebra $\mathfrak{so}(\infty)$ of the group of unitary operators 
on $L^2$ and, as before, generate an infinite sequence of functionally independent 
first integrals $\left\{ H_k \right\}_{k=1}^{\infty}$ in involution. 
This sequence corresponds to the infinite chain of embeddings 
$\mathfrak{so}(2)\subset \mathfrak{so}(3) \subset...\subset \mathfrak{so}(\infty)$.

Alternatively, one can introduce a decreasing sequence of Lie algebras
$
\mathfrak{so}(\infty)\supset \mathfrak{so}(\infty-1) \supset \mathfrak{so}(\infty-2) \supset \dots 
$
and construct the corresponding conserved quantities as follows.
Let $p^{(1)}$ be the orthogonal projection of $p \in L^2(M, d\mu)$ 
onto the hyperplane 
$e_1^\perp$. 
Let $p^{(2)}$ be the projection of $p$ onto the hyperplane 
$( e_1, e_2 )^\perp\subset L^2(M, d\mu)$ of codimension 2, etc. 
We define the $q^{(k)}$'s similarly and then set 
$$
H^{(k)}
= 
\int_M (p^{(k)})^2  d\mu \int_M (q^{(k)})^2 d\mu 
-
\left( \int_M q^{(k)}p^{(k)} \, d\mu \right)^2 
$$
for $k=0, 1, 2, \dots$. 
Note that $H^{(0)}$ coincides with the Hamiltonian $H=\int_M p^2\,d\mu$
on $S^\infty_1$. Furthermore, we have 

\begin{proposition}
The functionals $H^{(k)}$ are in involution, i.e., 
$\{H^{(k)},H^{(m)}\}=0$ for any $k,m=0,1,2,\dots$.
\end{proposition}
\begin{proof} Observe that the functional $H^{(0)}$ can be thought of as 
a Casimir invariant for the action of the unitary group in $L^2(M, d\mu)$, 
while $H^{(k)}$ is a Casimir for the action of 
the unitary group on the subspace $(e_1,...,e_k)^\perp\subset L^2(M, d\mu)$.
\end{proof}

Either one of the two sequences of functionals 
$\{ H_k \}_{k=1}^\infty$ or $\{ H^{(k)} \}_{k=1}^\infty$ 
constructed above 
provides infinitely many conserved quantities 
for the geodesic flow on the unit sphere $S^\infty_1 \subset L^2(M, d\mu)$. 

\begin{remark} \upshape 
As we mentioned in Remark \ref{rem:biham} 
the geodesic flow on the round sphere is bi-hamiltonian 
and completely integrable in a strong sense. 
This is true in the finite- as well as the infinite-dimensional cases. 
Nevertheless, it may still be of interest to exhibit explicitly 
the bi-hamiltonian structure of the Euler-Arnold equation \eqref{generalizedhuntersaxton} 
and, in particular, examine how the functionals constructed in this subsection 
correspond to the Lenard-Magri type invariants of the Hunter-Saxton equation on the circle $M=S^1$. 
\end{remark}

%%%%%%%%%%%%%%%%%%%%%%%%%%%%%%%%%%%%%%%%%%
%%%%%%%%%%%%%%%%%%%%%%%%%%%%%%%%%%%%%%%%%%
%%%%%%%%%%%%%%%%%%%%%%%%%%%%%%%%%%%%%%%%%%

\section{Global properties of solutions}
\label{global}
\nequation

Explicit formulas 
%\eqref{explicitgenhssoln}--\eqref{lagrangianhuntersaxtonsoln} 
of Theorem \ref{huntersaxtonsoln} make it possible 
to give a fairly complete picture of the global behavior of solutions to 
the $\dot H^1$ Euler-Arnold equation on $\VolM$ for any manifold $M$. 
It turns out for example that any smooth solution of this equation
 \eqref{generalizedhuntersaxton} 
has finite lifespan and the blowup mechanism can be precisely described. 

On the other hand, the fact that the Jacobian of the associated flow 
is reasonably well-behaved suggests a way of constructing global weak solutions. 
By the result of Moser~\cite{moser} the function on the right side of 
formula \eqref{lagrangianhuntersaxtonsoln} 
will be the Jacobian of some diffeomorphism as long as it is nowhere zero. 
In the first two subsections below we shall explore the possibility of constructing 
a smooth flow whose Jacobian may vanish at some points 
(as inevitably happens to the solutions past the blowup time). 
In Section \ref{fredholm} we discuss Fredholm properties of 
the associated Riemannian exponential map on $\VolM$. 

\subsection{Lifts to the diffeomorphism group}\label{weakmoser}
First, we note that there can be no global smooth (classical) solutions 
of the Euler-Arnold equation \eqref{generalizedhuntersaxton}. 
As in the case of the one-dimensional Hunter-Saxton equation 
all solutions break down in finite time.

\begin{proposition} \label{Tmax} 
The maximal existence time of a (smooth) solution of the Cauchy problem 
\eqref{generalizedhuntersaxton}-\eqref{HSic} 
constructed in Theorem \ref{huntersaxtonsoln} is 
\begin{equation}\label{blowuptime}
0< T_{\max} 
= 
\frac{\pi}{2\kappa} 
+ 
\frac{1}{\kappa} \arctan{\left( \frac{1}{2\kappa}\displaystyle\inf_{x\in M} \diver{u}_0(x) \right)}. 
\end{equation}
Furthermore, as $t \nearrow T_{\max}$ we have $\lVert u(t)\rVert_{C^1} \nearrow \infty$. 
\end{proposition} 
\begin{proof} 
This follows at once from formula \eqref{explicitgenhssoln} 
using the fact that $\diver{u}=\rho$. 
Alternatively, from formula \eqref{lagrangianhuntersaxtonsoln} 
we observe that the flow of $u(t,x)$ ceases to be a diffeomorphism at $t = T_{\max}$. 
\end{proof} 

Observe that before a solution reaches the blow-up time it is always possible to lift 
the corresponding geodesic to a smooth flow of diffeomorphisms using the classical construction 
of Moser \cite{moser}:

\begin{proposition}  \label{moser} 
There exists a family of smooth diffeomorphisms $\eta(t)$ in $\DiffM$ satisfying \eqref{lagrangianhuntersaxtonsoln}, i.e.
such that $\Jac_\mu(\eta(t)) = \varphi(t)$ where 
\begin{equation} \label{varphi} 
\varphi(t,x) = \Big( 
\cos{\kappa t} + \frac{\diver{u_0}(x)}{2\kappa} \sin{\kappa t}
\Big)^2\,,
\end{equation}
provided that $0 \leq t < T_{\max}$. 
\end{proposition} 
\begin{proof} 
Integrating $\varphi(t,x)$  over the manifold gives 
\begin{align*}
\int_M \varphi(t,x) \, d\mu 
=
\mu(M) \cos^2{\kappa t} 
&+ 
\frac{\sin{2\kappa t}}{2\kappa}  \int_M \diver{u_0}(x) \, d\mu       \\ 
&\qquad\qquad + 
\frac{\sin^2{\kappa t}}{4\kappa^2} \int_M (\diver{u_0}(x))^2 \, d\mu    
= \mu(M)\,,
\end{align*}
where we used formula \eqref{kappadef} for $\kappa$ and the fact that 
the integral of the divergence of $u_0$ vanishes. 
It follows that 
$$
\int_M \frac{\partial \varphi}{\partial t}(t,x) \, d\mu = 0
$$ 
which allows one to solve the equation 
$\Laplacian f = -\partial \varphi/\partial t$ 
for any fixed time $t$ and thus produce a function of two variables, e.g., 
\begin{equation}\label{fdef}
f(t,x) = f_1(x) \cos{2\kappa t} + f_2(x) \sin{2\kappa t}, 
\end{equation}
where $f_1$ and $f_2$ are smooth functions on $M$. 

Since for any $t$ in $[0,T_{\max})$ the function $\varphi(t,x)$ is strictly positive 
one can define a time-dependent vector field by 
$X(t,x) = \grad f(t,x)/\varphi(t,x)$. 
Letting $t \mapsto \xi(t)$ denote the flow of $X$ starting at the identity  
and using Cartan's formula and the definition of $f$ we now compute 
$$
\frac{d}{d t} \xi^*(\varphi \mu) 
= 
\xi^*\left( \frac{\partial \varphi}{\partial t}\mu  + \Lie_{X}(\varphi\mu)\right) = 0
$$
since $\Lie_X(\varphi \mu) = \diver{(\varphi X)}\mu$. 
Observing that $\varphi(0)= 1$ and $\xi(0)=e$ we conclude that 
$\xi^*(\varphi \mu) = \mu$ for any $0\leq t <T_{\max}$.
Since each $\xi(t)$ is a diffeomorphism, letting $\eta(t)$ denote its inverse 
we find that $\eta^*\mu = \varphi \mu$, 
from which it follows that $\Jac_\mu(\eta(t,x)) = \varphi(t,x)$ as desired. 
\end{proof} 

The method of Proposition \ref{moser} gives a particular choice of a diffeomorphism flow $\eta$ 
and hence a velocity field appearing in \eqref{generalizedhuntersaxton} 
and satisfying $\diver{u}=\varphi$. 
The flow must break down at the critical time $T_{\max}$ since the vector field $X$ 
becomes singular (when $\varphi$ reaches zero). 
The difficulty here is that one constructs $\eta$ indirectly, by first constructing $\xi=\eta^{-1}$,
and it is this inversion procedure that breaks down at the blowup time $T_{\max}$. 

%%%%%%%%%%%%%%%%%%%%%%%%%%%%%

\subsection{Heuristics beyond blow-up time} \label{Heuristics}
The geometric picture developed in Section \ref{aczero} helps to gain some insight 
into the blowup mechanism. 

The only functions in the image of the map 
$\Phi\colon \DiffM\to S^{\infty}_r$ 
constructed in Theorem \ref{isometricHSsphere} are those that are positive everywhere 
and the result of Moser \cite{moser} on transitivity of the action of diffeomorphisms on densities
shows that these are all such functions. 
%every such function is indeed in this image. 
But everywhere-positive functions form a relatively small open subset of $S^{\infty}_r$ 
and so any great circle that emanates from this set 
will eventually reach the antipodal point corresponding to a function 
which is everywhere negative (and thus not a Jacobian). 
On the other hand, the Jacobian of the flow $\eta$ is the \emph{square} 
of a function on the sphere (and hence it may still make sense). 

\begin{example} \upshape 
A simple toy model of this phenomenon can be constructed 
in finite-dimensional geometric statistics, see \cite{AN}. 
Consider a sample space of three outcomes $\{a,b,c\}$ 
with probabilities $P(a)=p$, $P(b)=q$ and $P(c)=r$ where $p+q+r=1$. 
The corresponding statistical model $\mathcal{S}$ is two-dimensional 
(its local coordinates are $p$ and $q$). 
Using the discrete analogue of the Fisher-Rao metric \eqref{fisherrao} 
the map 
$(p, q)\mapsto \big( 2\sqrt{p}, 2\sqrt{q},2\sqrt{1-p-q} \big)$ 
becomes an isometric embedding of $\mathcal{S}$ 
into the sphere $S^2$ in $\mathbb{R}^3$ of radius $2$.\footnote{Probability spaces 
are mapped into the positive octant of the sphere.} 
Consider a geodesic in $S^2$ starting at the uniform distribution $p=q=r=\frac{1}{3}$ 
\begin{equation}\label{spheregeod}
\gamma(t) 
= 
\left( \tfrac{2}{\sqrt{3}} \cos{t} 
+ 
\sqrt{2} \sin{t}, \tfrac{2}{\sqrt{3}} \cos{t} 
- 
\sqrt{2} \sin{t}, \tfrac{2}{\sqrt{3}} \cos{t} 
\right). 
\end{equation}

%%%%%%%%%%%%%
\begin{figure}
\bigskip\bigskip
     \begin{minipage}{0.45\textwidth}
 \begin{overpic}[width=\textwidth]{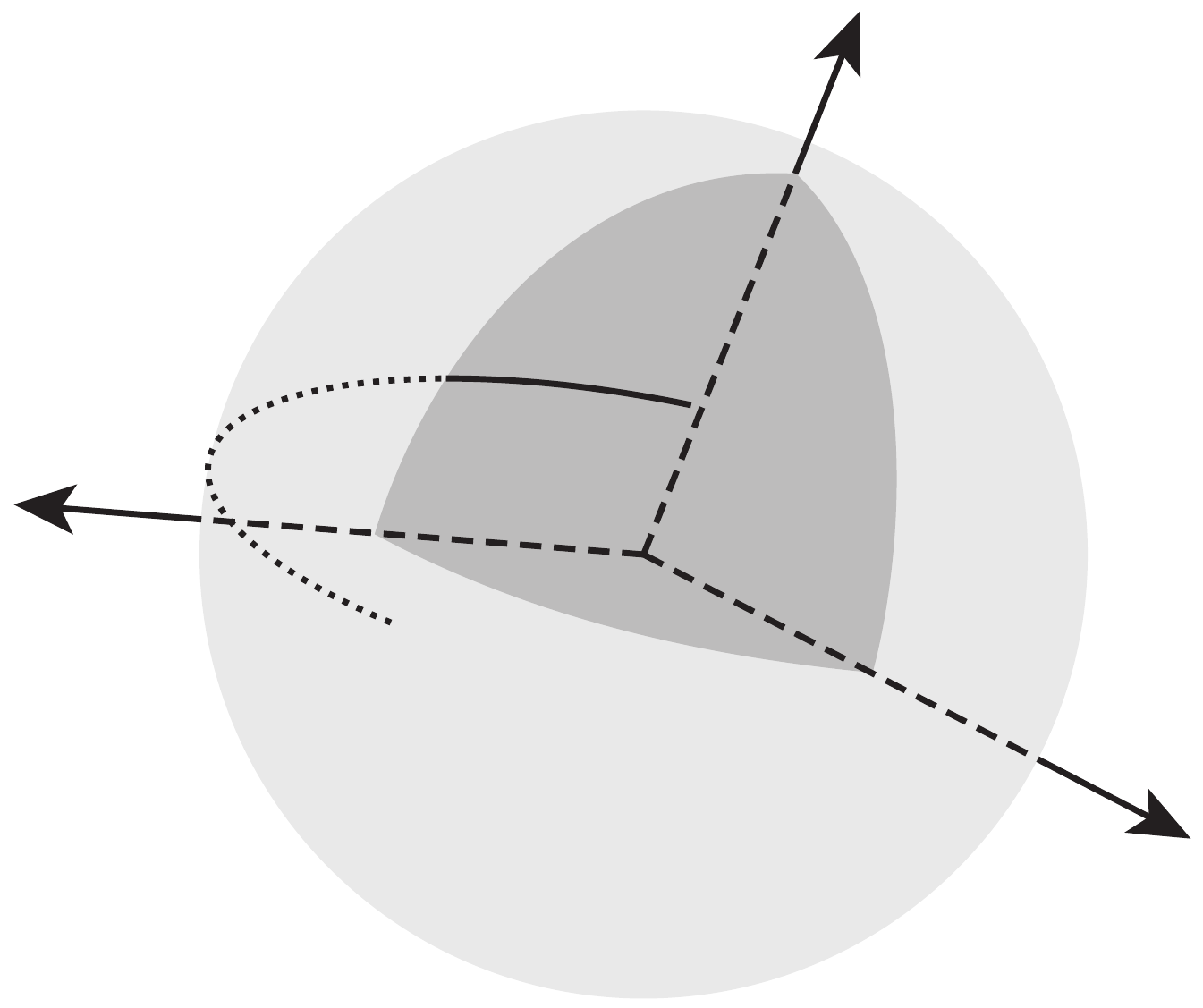}
      \put(3,33){\large $x$}
      \put(93,8){\large $y$}
      \put(72,86){\large $z$}
    \end{overpic}
     \end{minipage}
     \hspace{0.05\textwidth}
 \begin{minipage}{0.45\textwidth}
 \begin{overpic}[width=0.92\textwidth]{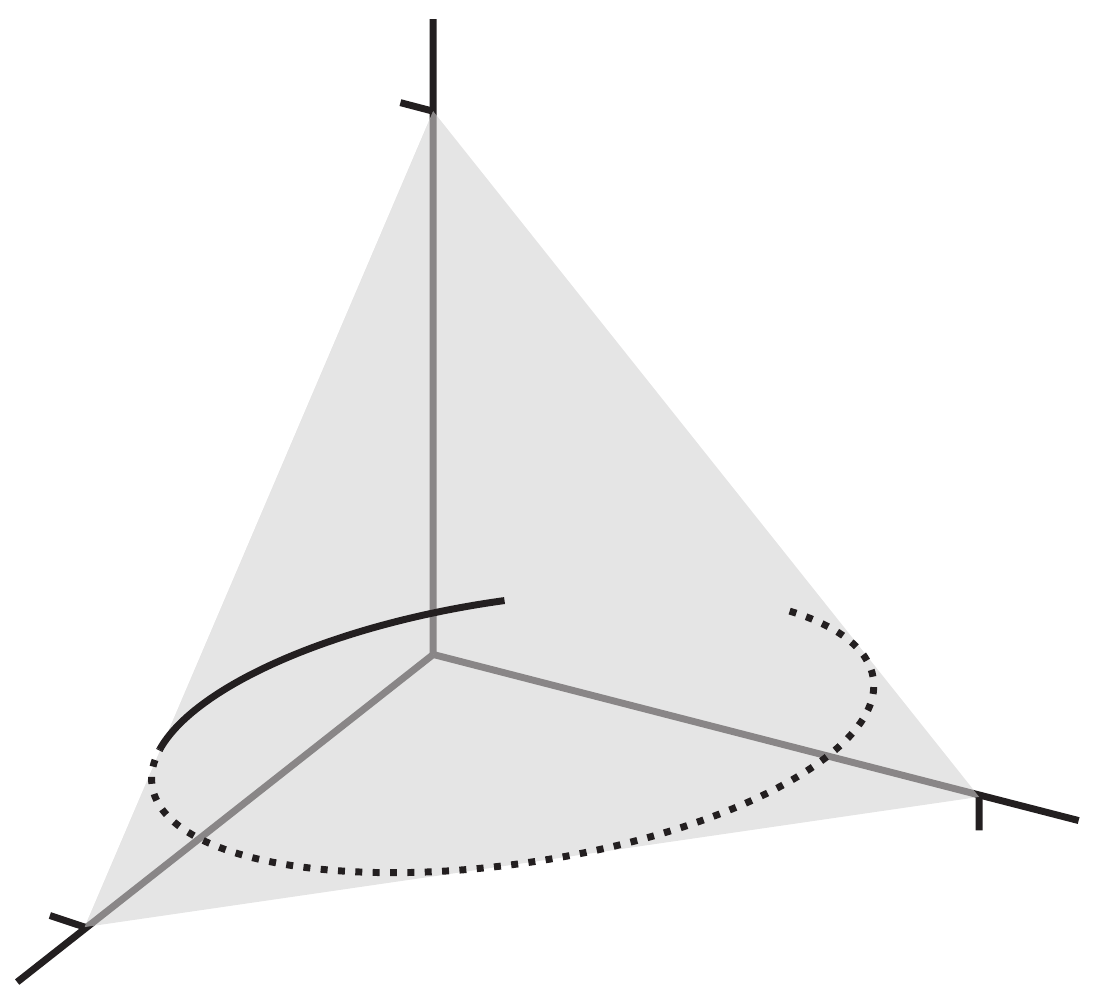}
      \put(8,0){\large $P(a)$}
      \put(90, 22){\large $P(b)$}
      \put(43,84){\large $P(c)$}
      \put(0,6){\large $1$}
      \put(88.3,9){\large $1$}
      \put(32,80){\large $1$}
    \end{overpic}
      \end{minipage}
\bigskip \bigskip    \begin{figuretext}\label{circlesphere}
The picture on the left shows a shaded region on a sphere representing the positive octant---the image of the square-root map on the probability space. A typical spherical geodesic leaves 
this region in finite time. 
The picture on the right shows a geodesic in the probability space given by squaring each component; 
when the geodesic hits the boundary it ``bounces off the wall.''
              \end{figuretext} 
\end{figure}
%%%%%%%%%%%%%

The corresponding probabilities are 
\begin{align}\label{probgeod} 
P(a) = \tfrac{1}{4} \left( \tfrac{2}{\sqrt{3}} \cos{t} + \sqrt{2} \sin{t} \right)^2, 
\quad  
&P(b) = \tfrac{1}{4} \left( \tfrac{2}{\sqrt{3}} \cos{t} - \sqrt{2} \sin{t} \right)^2,   
\nonumber  \\ 
&P(c) = \tfrac{1}{3} \cos^2{t}
\end{align}
Thus ``global'' geodesics of the Fisher-Rao metric ``bounce off the walls''
of the convex triangle as the individual probabilities approach zero and then increase again;
see Figure \ref{circlesphere}.
\end{example} 

\begin{example} \upshape 
For the Hunter-Saxton equation on $\Diff(S^1)/\Rot(S^1)$ a similar phenomenon 
was explained in \cite{L3}. 
In this case 
the flow $\eta$ is determined (up to rotations of the base point) by its Jacobian. 
Hence, spherical geometry ``selects'' a unique (weak) global solution 
$\eta(t)\in \Diff(S^1)/\Rot(S^1)$ which is always smooth. 
If the initial velocity is not constant in any interval, then the singularities of $\eta$ are isolated 
so that $\eta$ is a homeomorphism (but not a diffeomorphism past the blowup time). 
This may be interpreted as ``elastic collisions'' of particles (energy is conserved). 
\end{example} 

In higher dimensions if the Jacobian is not everywhere positive the situation 
is much more complicated. 
Nevertheless, in this case it may be possible to apply the techniques of 
Gromov and Eliashberg \cite{GE} or Cupini et al.~\cite{CDK} in order to construct a map with a prescribed Jacobian. Namely, let $\varphi(t,x)$ be defined by the right side of \eqref{varphi}. 
Then, for any fixed time $t$, there is a map 
$\eta(t, x)$ of $M$ such that 
$\Jac_\mu \eta(t,x) = \varphi(t,x)$ 
even if $\varphi(t,x)$ is zero at some $x$. 
Furthermore, using for example results of \cite{CDK}, 
the map $\eta$ can be chosen to be $C^{\infty}$ smooth provided that $\varphi$ is  smooth 
so that by 
formula \eqref{lagrangianhuntersaxtonsoln} 
one may define a ``global flow'' in the space $C^{\infty}(M, M)/\DiffmuM$. 
It would be interesting to extend Moser's argument to construct 
a global flow of homeomorphisms out of this flow of maps (past the blowup time).

\subsection{The $\dot{H}^1$-exponential map of $\DiffM/\DiffmuM$}\label{fredholm}

We next describe the structure of singularities of the exponential map 
of our right-invariant $\dot{H}^1$-metric on the space of densities.

First, recall from Proposition \ref{diameter} that the diameter of $\VolM$ 
with respect to the metric \eqref{multiHSmetric} is equal to $\pi\sqrt{\mu(M)}/2$. 

\begin{corollary}
Any geodesic in $\VolM=\DiffM/\DiffmuM$ starting from the reference density 
is free of conjugate points.
\end{corollary}
\begin{proof} 
Indeed, for any geodesic on a sphere the first conjugate point to the initial point is its antipode. 
However, 
the great arcs contained in the image of the space of densities $\VolM$ 
under the isometry  $\Phi$ constructed in Theorem \ref{isometricHSsphere} 
cannot contain antipodal points along them 
because the diameter of this image is $\pi\sqrt{\mu(M)}/2$. 
\end{proof} 

Using the techniques of \cite{MP} one can show that the Riemannian exponential map
of \eqref{multiHSmetric} on $\VolM$ is a nonlinear Fredholm map. 
In other words, its differential is a bounded Fredholm operator 
(on suitable Sobolev completions of tangent spaces) 
of index zero 
for as long as the solution is defined. 
The fact that this is true for the general right-invariant $a$-$b$-$c$ metric 
given at the identity by \eqref{abcmetric} on $\DiffM$ or $\DiffmuM$ 
also follows from the results of \cite{MP}. 
More precisely, we have the following 

\begin{theorem} 
For a sufficiently large Sobolev index $s>n/2 + 1$ the Riemannian exponential map of 
\eqref{multiHSmetric} on the quotient $\Diff^s(M)/\Diff_\mu^s(M)$ 
of the $H^s$ completions 
is Fredholm up to the blowup time $t=T_{\max}$ given in \eqref{blowuptime}.
\end{theorem}

The proof of Fredholmness given in \cite{MP} is based on perturbation techniques. 
The basic idea is that the derivative of the exponential map along any geodesic 
$t\mapsto \eta(t) = \exp_e(tu_0)$ can be expressed as 
$ 
\exp_{e\ast tu_0} = t^{-1} dL_{\eta(t)} \Psi(t),
$
where $\Psi(t)$ is a time dependent operator satisfying the equation 
\begin{equation}\label{lambdaK}
\Psi(t) 
= 
\int_0^t \Lambda(\tau)^{-1}\,d\tau 
+ 
\int_0^t \Lambda(\tau)^{-1} B\big(u_0, \Psi(\tau)\big) \, d\tau 
\end{equation}
and where $\Lambda = \Ad_{\eta}^*\Ad_{\eta}$ (as long as $t <T_{\max})$.
If the linear operator $w\mapsto B(u_0,w)$ is compact for any sufficiently smooth $u_0$ 
then $\Psi(t)$ is Fredholm being a compact perturbation of 
the invertible operator defined by the integral 
$\int_0^t \Lambda(\tau)^{-1}\,d\tau$. 
In the same way one can check that this is indeed the case for 
the homogeneous space of densities 
with the right-invariant metric \eqref{multiHSmetric}. 
We will not repeat the argument here and refer to \cite{MP} for details. 

\begin{remark}\upshape 
We emphasize that the perturbation argument described above 
works only for sufficiently short geodesic segments in the space of densities. 
Recall that for the round sphere in a Hilbert space 
the Riemannian exponential map cannot be Fredholm for sufficiently long geodesic
because any geodesic starting at one point has a conjugate point of infinite order 
at the antipodal point. 
In the case of the metric \eqref{multiHSmetric} on the space of densities 
one checks that 
$\lVert \Lambda(t)^{-1}\rVert \nearrow \infty$  as $t\nearrow T_{\max}$
since it depends on the $C^1$ norm of $\eta$ via the adjoint representation.
Therefore the argument of \cite{MP} breaks down here past the blowup time as equality \eqref{lambdaK} 
becomes invalid.\footnote{It is tempting to interpret this phenomenon as 
the infinite multiplicity of conjugate points on the Hilbert sphere \emph{forcing} 
the classical solutions of \eqref{generalizedhuntersaxton}
to break down before the conjugate point is reached.} 
\end{remark}

%%%%%%%%%%%%%%%%%%%%%%%%%%%%%%%%%%%%%%%%%
%%%%%%%%%%%%%%%%%%%%%%%%%%%%%%%%%%%%%%%%%
%%%%%%%%%%%%%%%%%%%%%%%%%%%%%%%%%%%%%%%%%%%%% 

\section{The general Euler-Arnold equation of the $a$-$b$-$c$ metric}
\label{abcmetricsection}
\nequation

In this section we compute the general Euler-Arnold equation for 
the $a$-$b$-$c$ metric \eqref{abcmetric} on the full diffeomorphism group $\Diff(M)$. 
We will be particularly interested in the degenerate cases 
when $a=0$ and either $b=0$ or $c=0$ 
because they lead naturally to nondegenerate metrics 
on the volumorphism group $\Diffmu(M)$ 
or the homogeneous space of densities $\VolM$ respectively. 

It is convenient to proceed with the derivation of the 
Euler-Arnold equation in the language 
of differential forms. We begin with a brief comment regarding the notation. 
Recall that the exterior derivative operator $d$ on $k$-forms 
and its adjoint $\delta$ are related by 
$$
\delta \beta = (-1)^{nk +1}*d*\beta, 
\qquad 
\beta \in \Omega^{k+1}(M),
$$
where the Hodge star operator $\ast$ is defined with respect to 
our fixed Riemannian volume element $\mu$ on $M$ 
by the formula 
$
\alpha \wedge \ast \beta 
= 
\langle \alpha, \beta \rangle \, \mu
$
for any $\alpha$ and $\beta$ in $\Omega^k(M)$. 
As usual, the symbols $\flat$ and $\sharp = \flat^{-1}$ denote the isomorphisms 
between vector fields and one-forms induced by the Riemannian metric on $M$. 
While we use $d$ and $\delta$ notations throughout, 
we will continue to employ the more familiar formulas when available. 
For example, in any dimension we have 
$\delta u^{\flat} = -\diver{u}$, 
if $n=1$ then $du^{\flat}=0$, while if $n=3$ then $du^{\flat} = \ast\curl{u}^{\flat}$ 
and $\delta du^{\flat} = (\curl^2{u})^{\flat}$ for any vector field, etc. 
In particular, in dimension $n=3$ the $a$-$b$-$c$ metric \eqref{abcmetric} 
assumes the following special form
$$
\llangle u,v\rrangle 
= 
a \int_M \langle u,v\rangle \, d\mu 
+ 
b\int_M \diver{u}\, \diver{v} \, d\mu 
+ 
c\int_M \langle \curl{u}, \curl{v}\rangle \, d\mu. 
$$
In dimension $n=2$ we have $\curl{u}$ defined by the formula 
$du^{\flat} = (\curl{u}) \mu$ 
and therefore 
$\langle \curl{u}, \curl{v}\rangle=\curl{u}\, \curl{v}$ 
is understood as a product of functions. 
For $n=1$ the metric \eqref{abcmetric} simplifies to
$$
\llangle u,v\rrangle 
= a\int_{S^1} uv \,dx + b \int_{S^1} u_x v_x \, dx \,.
$$

Recall also that the (regular) dual $T_e^*\Diff(M)$ of the Lie algebra $T_e\Diff(M)$ 
admits the orthogonal Hodge decomposition\footnote{Orthogonality of the components 
in \eqref{hodgedecomposition} is established for suitable Sobolev completions 
with respect to the induced metric on differential forms 
$\llangle \alpha^{\sharp}, \beta^{\sharp} \rrangle$.} 
\begin{equation}\label{hodgedecomposition}  
T_e^*\Diff(M) = d\Omega^0(M) \oplus \delta \Omega^2(M) \oplus \mathcal{H}^1,
\end{equation} 
where $\Omega^{k}(M)$ and $\mathcal{H}^k$ denote the spaces of smooth $k$-forms 
and harmonic $k$-forms on $M$, respectively; see Figure \ref{hodge.pdf}.

\begin{figure} 
\begin{center}
  \begin{overpic}[width=.6\textwidth]{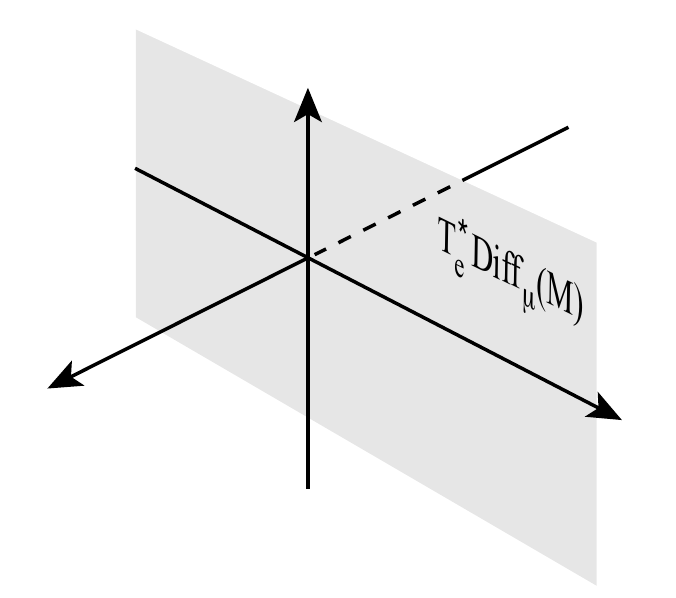}
 %   \put(-20,75){\large $T_e^*\DiffM$}
    \put(43,78){\large $\mathcal{H}^1$}
  %  \put(37,79){\large $a$-term}
    \put(-12,18){\large $b$-term}
    \put(-13,25){\large $d\Omega^0(M)$}
     \put(95,23){\large $\delta \Omega^2(M)$}
     \put(96,17){\large $c$-term}
 \end{overpic}
\begin{figuretext} \label{hodge.pdf}
The Hodge decomposition of $T_e^*\Diff(M)$. 
The $b$ and $c$ terms of the metric (\ref{abcmetric}) 
are nondegenerate along the axes indicated in the figure.
\end{figuretext}
\end{center}
\end{figure}

We now proceed to derive the Euler-Arnold equation of 
the $a$-$b$-$c$ metric \eqref{abcmetric}. 
Let $A:T_e\DiffM \to T_e^*\DiffM$ be the self-adjoint elliptic operator 
\begin{equation} \label{ellA}
Av = av^\flat + b d \delta v^\flat + c \delta dv^b 
\end{equation} 
(the inertia operator) so that 
\begin{equation}\label{Adef}
\llangle u, v \rrangle = \int_M \langle Au, v \rangle \, d\mu, 
\end{equation}
for any pair of vector fields $u$ and $v$ on $M$.

%One could also consider other subgroups such as the symplectomorphism or 
%contactomorphism groups; \cite{sm}. 
%\end{rem}
%In general, note that when $a=0$ the metric is degenerate and yields 
%a genuine Riemannian geometry 
%only on a quotient space. 

%Note that $\mathcal{H}^1$ is isomorphic to 
%the first cohomology group of $M$, so that the topology of the manifold $M$ 
%is encoded in the last factor of (\ref{hodgedecomposition}).
%
%Let us make a  few comments on the three terms in \eqref{abcmetric}. 
%The first term is the usual $L^2$ inner product, which is non-degenerate 
%on $T_{id}^*\Diff$. 
%The remaining terms are non-degenerate only on $d\Omega^0$ and 
%$\delta \Omega^2$ respectively, 
%thus providing $H^1$ contributions along these two orthogonal subspaces 
%in \eqref{hodgedecomposition}.
%A geometric explanation of the origin of these two terms will be given 
%in Section \ref{sec:Met-orbits}. 

\begin{proposition}\label{generalBcomputation}
The Euler-Arnold equation of the Sobolev $a$-$b$-$c$ metric \eqref{abcmetric} 
on $\DiffM$ has the form 
\begin{align} \label{generalgeodesic}
Au_t 
= 
-a\left( (\diver{u}) \, u^\flat + \iota_u du^\flat + d \langle u, u \rangle \right) 
&- b\left( (\diver{u}) \, d \delta u^\flat  + d \iota_u d \delta u^\flat \right)
\nonumber	\\
&-  c \left( (\diver{u}) \, \delta du^\flat  + \iota_u d\delta du^\flat 
+ 
d\iota_u \delta du^\flat \right) 
\end{align}
where $A$ is given by \eqref{ellA}. 
\end{proposition}

\begin{proof} 
By definition \eqref{opB} of the bilinear operator $B$, 
for any vectors $u, v$ and $w$ in $T_e \DiffM$ we have 
\begin{equation}\label{BdefA} 
\llangle B(u,v), w\rrangle 
= 
\llangle u, \ad_vw\rrangle = -\int_M \langle Au, [v,w]\rangle \, d\mu.
\end{equation}
Integrating over $M$ the following identity 
$$ 
\langle Au, [v,w]\rangle 
= 
\big\langle d \langle Au,w\rangle, v\big\rangle 
- 
\big\langle d \langle Au, v\rangle, w \big\rangle 
- 
dAu(v, w) 
$$ 
and using 
$$
\int_M \big\langle d \langle Au,w\rangle, v\big\rangle \, d\mu 
= 
-\int_M  \langle Au, w\rangle \diver{v} \, d\mu 
$$
we get 
$$ 
\llangle u, \ad_v w \rrangle 
= 
\int_M 
\big\langle 
(\diver{v}) Au + d \langle Au, v\rangle 
+  \iota_vdAu, w 
\big\rangle \, d\mu.
$$

On the other hand, we have 
$$
\llangle B(u,v), w\rrangle = \int_M \langle AB(u,v), w\rangle \, d\mu 
$$ 
and, since $w$ is an arbitrary vector field on $M$, 
comparing the two expressions above, we obtain 
\begin{equation}\label{generalBop}
B(u,v) 
= 
A^{-1}\big( (\diver{v})Au + d\langle Au, v\rangle 
+ \iota_vdAu\big).
\end{equation}

Setting $v=u$, isolating the coefficients $a$, $b$, and $c$, 
and using \eqref{utBuu} yields the equation \eqref{generalgeodesic}. 
The simplification in the $b$ term comes from $d^2=0$.
\end{proof}

\begin{remark} \label{rem:EAeqns} \upshape 
Special cases of the Euler-Arnold equation \eqref{generalgeodesic} 
include several well-known evolution PDE. 
\begin{itemize}
\item For $n=1$ and $a=0$, we obtain the Hunter-Saxton equation 
\begin{equation*}  %\tag{\ref{HS}} 
u_{txx} + 2u_x u_{xx} + uu_{xxx} = 0. 
\end{equation*} 
\item For $n=1$ and $b=0$, we get the (inviscid) Burgers equation 
\begin{equation*}  %\tag{\ref{invB1D}} 
u_t + 3uu_x=0. 
\end{equation*} 
\item For $n=1$ and $a=b=1$, we obtain the Camassa-Holm equation 
\begin{equation*}  %\tag{\ref{CH1D}} 
u_t - u_{txx} + 3uu_x - 2u_x u_{xx} - u u_{xxx} = 0.
\end{equation*} 
\item For any $n$ when $a=1$ and $b=c=0$ we get the multi-dimensional 
(right-invariant) Burgers equation 
\begin{equation*}  \label{riBurgers} 
u_t + \nabla_u u + u\diver{u} + \frac{1}{2}\grad \lvert u\rvert^2 = 0; 
\end{equation*} 
(it is sometimes rewritten using 
$(\iota_udu^{\flat})^{\sharp} = \nabla_uu - \frac{1}{2} \grad \lvert u\rvert^2$ 
and referred to as the template matching equation), 
see \cite{HMA}. 
\item For any $n$ and $a=b=c=1$ we get the EPDiff equation
\begin{equation*} 
m_t + \mathcal{L}_um + m\diver u = 0, 
\qquad 
m = u^\flat - \Delta u^\flat, 
\end{equation*} 
see e.g., \cite{HMR}.
\end{itemize}
\end{remark}

%%%%%
Now observe that if $a=0$ then the $a$-$b$-$c$ metric 
becomes degenerate and can only be viewed as a (weak) Riemannian metric 
when restricted to a subspace. 
There are three cases to consider. 
\begin{enumerate}
\item $a=0, b\ne 0, c=0$: 
the metric is nondegenerate on the homogeneous space $\VolM=\Diff(M)/\Diffmu(M)$ 
which can be identified with the space of volume forms or densities on $M$. 
This is our principal example of the paper, 
studied in Sections \ref{aczero}, \ref{statistics}, \ref{geodesics} and \ref{global}. 
\item $a=0, b=0, c\ne 0$: 
the metric is nondegenerate on the group of (exact) volumorphisms 
and the Euler-Arnold equation is \eqref{exactogeodesic}, 
see Corollary \ref{euler-like}  
and Remark \ref{rem:longtime} below.
\item $a=0, b\ne 0, c\ne 0$: 
the metric is nondegenerate on the orthogonal complement of the harmonic fields. 
This is neither a subalgebra nor the complement of a subalgebra in general 
and thus the approach of taking the quotient modulo a subgroup 
developed in the other cases cannot be applied here. 
However, in the special case when $M$ is the flat torus $\mathbb{T}^n$ 
the harmonic fields are the Killing fields 
which 
do form a subalgebra (whose subgroup $\Isom(\mathbb{T}^n)$ consists of the isometries). 
In this case we get a genuine Riemannian metric on the homogeneous space 
$\Diff(\mathbb{T}^n)/\Isom(\mathbb{T}^n)$. 
\end{enumerate}
%
%Metrics on the exact volumorphism group (of which \eqref{abcmetric} when $a=b=0$ 
%is a special case) were studied in \cite{MP}. 
%In this case the geodesic equation is given by 
%\begin{equation}\label{conlygeodesic}
%d\delta du^{\flat} + d\iota_u d\delta d u^{\flat} = 0.
%\end{equation}

In cases (1) and (3) above one needs to make sure that the degenerate (weak Riemannian) metric 
descends to a non-degenerate metric on the quotient. This can be verified using 
the general condition \eqref{descent} in Proposition \ref{KhesinMisiolekprop}. 

The corollary below contains Proposition \ref{generalizedhuntersaxtonprop} 
which was stated earlier without proof. 

\begin{corollary} \label{descentmetricvolum}
If $a=c=0$ then the $a$-$b$-$c$ metric \eqref{abcmetric} 
satisfies condition \eqref{descent} and therefore 
descends to a metric on the space of densities $\VolM$. 
The corresponding Euler-Arnold equation is 
\begin{equation}\label{homgeodesic}
d\diver{u}_t + \diver{u}\, d\diver{u} + d \iota_u (d\diver{u}) = 0 
\end{equation}
or, in the integrated form, 
\begin{equation} \tag{\ref{generalizedhuntersaxton}}
\rho_t + \langle u, \nabla \rho \rangle 
+ 
\tfrac{1}{2} \rho^2 
=
- \frac{\int_M \rho^2 \, d\mu}{2\mu(M)} 
\end{equation}
where $\rho = \diver{u}$. 
\end{corollary}
\begin{proof} 
We verify \eqref{descent} for $G=\DiffM$, $H=\DiffmuM$ 
and $\ad_wv = -[w,v]$, where $[\cdot, \cdot]$ is the Lie bracket of vector fields on $M$. 
Given any vector fields $u,v$ and $w$ with $\diver{w} = 0$, we have 
\begin{align*}
\llangle \ad_wv, u\rrangle_{\dot{H}^1} + \llangle v,  \ad_wu\rrangle_{\dot{H}^1} 
&= -b\int_M \big( 
\diver{[w,v]} \diver{u} + \diver{[w,u]} \diver{v} 
\big) \, d\mu \\  
&= -b\int_M \Big(
\big( \langle w , \nabla\diver{v} \rangle 
- 
\langle v, \nabla\diver{w} \rangle \big) \diver{u}  \\
&\qquad\quad\quad\quad +
\big( \langle w, \nabla\diver{u} \rangle 
- 
\langle u, \nabla\diver{w} \rangle \big) \diver{v} 
\Big) \, d\mu \\
%&\qquad\qquad= 
%-b\int_M 
%\big\langle w, \diver{u}\nabla\diver{v} + \diver{v}\nabla\diver{u} \big\rangle 
%\, d\mu \\
&= b\int_M \diver{w}\cdot  \diver{v} \cdot\diver{u} \, d\mu  = 0 \,,
\end{align*}
which shows that \eqref{abcmetric} descends to 
$\Diff(M)/\Diff_\mu(M)$.\footnote{Note that the same proof works for the quotient space 
$G/H=\Diff(\mathbb{T}^n)/\Isom(\mathbb{T}^n)$.}

The Euler-Arnold equation on the quotient can be now obtained 
from \eqref{generalgeodesic} in the form (\ref{homgeodesic}). 
In integrated form it reads 
$$ 
\diver{u_t} + u( \diver{u}) + \tfrac{1}{2} (\diver{u})^2 = C(t) 
$$
where $C(t)$ may in general depend on time. 
Integrating this equation over $M$ determines the value of $C(t)$. 
%
%Using Khesin-Misio{\l}ek~\cite{K-M}, the geodesic equation of 
%an invariant metric on a
%homogeneous space is derived in the same way as on a Lie group:
%variations are of the form 
%$$
%\frac{\partial u}{\partial s}\big\vert_{s=0} 
%= 
%y_t + [u,y] = y_t-\ad_uy
%$$ 
%for a variation field $y$. If the action is to be stationary,
%then we must have
%$$ \int_0^T \llangle y_t - \ad_uy, u\rrangle \, dt = 0$$ for every
%variation field $y$ vanishing at $t=0$ and $t=T$. Using the explicit
%formula \eqref{multiHSmetric}, we get
%\begin{align*}
%0 &= \int_0^T \int_M \diver{(y_t+[u,y])} \diver(u) \, \mu\,dt\\ &=
%-\int_0^T \int_M \diver{y} \diver{u_t} \, \mu \, dt + \int_0^T
%\int_M \big( d\diver(y) \cdot u - d\diver(u) \cdot y\big) \diver{u}\, \mu \,
%dt\\
%&= -\int_0^T \int_M \diver(y) \left(  \diver{u_t} +
%\diver{\big(\mathcal{L}_u(\diver{u})\big)} + \frac{1}{2} (\diver{u})^2 \right)
%\, \mu \, dt.
%\end{align*}
%Since this must be true for every vector field $y$, we must have
%$$ \diver{u}_t + \mathcal{L}_u(\diver{u}) + \frac{1}{2} (\diver{u})^2 = -C$$
%for some constant $C$ (possibly depending on time, but not on
%the spatial variables). Finally, since $\int_M \diver{u} \,
%\mu = 0$ for all time, integration over $M$ determines the value of $C$.
\end{proof}
\smallskip

%%%%%%% 
We now return to the nondegenerate $a$-$b$-$c$ metric ($a\not=0$) and 
restrict it to the subgroup of volumorphisms (or exact volumorphisms), cf. Figure \ref{hodge.pdf}. Observe that one obtains the corresponding Euler-Arnold equations  with $b=0$ 
directly from \eqref{generalgeodesic} using appropriate Hodge projections. 

\begin{corollary} \label{projectiongeodesic} 
The Euler-Arnold equation of the $a$-$b$-$c$ metric \eqref{abcmetric} 
restricted to the subgroup $\Diffmu(M)$ has the form 
\begin{equation}\label{volumegeodesicexplicit}
au_t^\flat + c\delta d u_t^\flat + a \iota_udu^\flat +  c \iota_ud\delta du^\flat 
= 
d\Delta^{-1}\delta \left( a \iota_u du^\flat + c \iota_u d\delta du^\flat \right). 
\end{equation} 
\end{corollary}
\begin{proof} 
Let $P$ denote the orthogonal Hodge projection of 
(Sobolev completions of) the tangent space $T_e\DiffM$ onto the tangent space 
$T_e\DiffmuM$ to the subgroup at the identity.  
In this case the Euler-Arnold equation reads\footnote{Similar computations 
in the general case of a projection of an Euler-Arnold equation to a subgroup 
can be found in \cite{MP} or \cite{taylor}.} 
$$
u_t = - P B(u,u) = -P A^{-1} \left( \delta u^\flat Au +  d ( Au(u) ) + \iota_u dAu \right) 
$$ 
where we used \eqref{generalBop}. 

Since $\delta u^\flat =0$ and since the Hodge projection 
$P = I - d\Delta^{-1}\delta$ 
commutes with the (inertia) operator $A$ we have 
$$
Au_t = -P \iota_u dAu 
= 
-\iota_u dAu + d\Delta^{-1}\delta \iota_u dAu 
$$
which is an equivalent form of \eqref{volumegeodesicexplicit}. 
\end{proof} 

One should also mention that 
\eqref{volumegeodesicexplicit}
is also the Euler-Arnold equation when the $a$-$b$-$c$ metric \eqref{abcmetric} is restricted to 
the set of exact volumorphisms $\Diffmuex(M)$, i.e. the subgroup of $\DiffM$ 
whose tangent space $T_e \Diffmuex(M)$ at the identity consists of 
vector fields $v$ of the form 
$v^\flat = \delta\beta$, 
where $\beta$ is a $2$-form on $M$. 
In this case the corresponding orthogonal projection is 
$P = \delta\Laplacian^{-1} d$. 

\begin{remark} \label{rem:euler-a} \upshape 
It is interesting to note that the Euler-Arnold equation \eqref{volumegeodesicexplicit} 
is closely related to the $H^1$ Euler-$\alpha$ equation 
\begin{equation} \label{euleralfa} 
(u^{\flat}+\alpha^2\delta du^{\flat})_t 
+ 
\mathcal{L}_u (u^{\flat} + \alpha^2 \delta du^{\flat}) = 0, 
\qquad\quad 
(\alpha \in \mathbb{R}) 
\end{equation} 
which was proposed as a model for large-scale motions 
by Holm, Marsden and Ratiu \cite{HMR}. 
\end{remark} 

There is also a ``degenerate analogue'' of the latter equation which corresponds to the case 
where $a=b=0$:

\begin{corollary} \label{euler-like} 
The Euler-Arnold equation of the right-invariant metric \eqref{abcmetric} 
with $a=b=0$ 
on the subgroup of exact volumorphisms is 
\begin{equation}\label{exactogeodesic} 
\delta du^\flat_t + P\mathcal{L}_u (\delta  du^\flat) = 0, 
\end{equation}
where $P$ is the orthogonal Hodge projection onto 
$\delta\Omega^2(M)$. 
\end{corollary} 
\begin{proof} 
Set $a=b=0$ in Equation \eqref{generalgeodesic} and project onto $\delta \Omega^2(M)$. 
\end{proof} 

\begin{remark}\label{rem:longtime} \upshape 
Observe that the nonlinear PDE in \eqref{exactogeodesic} 
has a simpler structure than the Euler-$\alpha$ equation \eqref{euleralfa}. 
In fact, one can view equation \eqref{exactogeodesic} as a formal limit of 
\eqref{euleralfa} when $\alpha\nearrow\infty$. 
Furthermore, it seems likely that establishing global well-posedness for either equation 
requires overcoming very similar technical obstacles 
and we expect that standard techniques will be sufficient to obtain 
global well-posedness of \eqref{exactogeodesic} in the two-dimensional case. 
%On the other hand, an analogue of the Beale-Kato-Majda criterion 
%$\int_0^T \lVert \curl{u}(t)\rVert_{BMO} \, dt < \infty$ for every $T>0$ 
%will be a necessary and sufficient condition in three dimensions 
%(compare a recent paper \cite{HL} on the Euler-$\alpha$ equation). 
%However, we leave the details for another publication. 
\end{remark}

%%%%%%%%%%%%%%%%%%%%%%%%%%%%%%%%%%%%%%%%
%%%%%%%%%%%%%%%%%%%%%%%%%%%%%%%%%%%
%%%%%%%%%%%%%%%%%%%%%%%%%%%%%%%%%%%%%%%%%%%%%%%

\section{The space of metrics and the diffeomorphism group} 
\label{sec:Met-orbits} 
\nequation

Our aim in this section is primarily motivational. 
Apart from the fact that the Euler-Arnold equations of $H^1$ metrics 
yield a number of interesting evolution equations 
of mathematical physics discussed above
there is also a purely geometric reason to study them. 
Below, we  show that right-invariant Sobolev metrics 
of the type  studied in this paper arise naturally on orbits of 
the diffeomorphism group acting on the space of all Riemannian metrics and volume forms on $M$. 
First we collect a few basic facts in Section \ref{subsec:efg}. 

%%%%%%%%%%%%%
\subsection{Weak $L^2$ structures on the spaces of metrics and volume forms} 
\label{subsec:efg} 
Our main references for the constructions recalled in this subsection are
\cite{be, ebin, fg}. 

Given a compact manifold $M$ consider the set $\Met(M)$ of 
all (smooth) Riemannian metrics on $M$. 
This set acquires in a natural way the structure of 
a smooth Hilbert manifold.\footnote{Indeed, the closure of $C^{\infty}$ metrics 
in any Sobolev $H^s$ norm with $s>n/2$ is an open subset of $H^s(S^2T^\ast M)$.} 

The group $\Diff(M)$ acts on $\Met(M)$ by pull-back 
$g\mapsto \pullbackg(\eta)=\eta^*g$ 
and 
there is a natural geometry on $\Met(M)$ which is invariant under this action. 
If $g$ is a Riemannian metric and $A, B$ are smooth sections of 
the tensor bundle $S^2T^\ast M$, then the expression 
\begin{equation}\label{metricofmetrics} 
\llangle A, B\rrangle_g = \int_M \Tr_g(AB) \, d\mu_g 
\end{equation}
defines a (weak Riemannian) $L^2$-metric on $\Met(M)$. 
Here $\mu_g$ is the volume form of $g$ and 
the trace is defined by 
$
%\begin{equation} \label{eq:trace}
\Tr_g(AB) 
:=
\Tr \big( g^{-1}A \, g^{-1}B \big)\,. 
$
This metric is invariant under the action of $\Diff(M)$, see \cite{ebin}.

The space $\mathrm{Vol}(M)$ of all (smooth) volume forms on $M$ 
also carries a natural (weak Riemannian) $L^2$-metric  
\begin{equation}\label{volumemetric}
\llangle \alpha,\beta\rrangle_{\nu} 
= 
\frac{4}{n} \int_M \frac{d\alpha}{d\nu}\, \frac{d\beta}{d\nu} \, d\nu,
\end{equation}
where $\nu \in \mathrm{Vol}(M)$ and $\alpha, \beta$ are smooth $n$-forms 
and which appeared already in the paper \cite{fg}.\footnote{The space $\mathrm{Vol}(M)$ 
of volume forms on $M$ contains the codimension 1 submanifold
 $\VolM\subset \mathrm{Vol}(M)$ of those forms whose total volume is normalized.}
It is also invariant under the action of $\Diff(M)$ by pull-back 
$\mu \to \pullbackmu(\eta)=\eta^\ast\mu$. 

There is a map $\Volumizer \colon \MetM \to \mathrm{Vol}(M)$ 
which assigns to a Riemannian metric $g$ the volume form $\mu_g$. 
Its derivative
$\Volumizer_{*g}\colon T_g\MetM \to T_{\Volumizer(g)}\mathrm{Vol}(M)$ 
is 
$$
\Volumizer_{*g}(A) = \Tr_{\Volumizer(g)}A = \Tr \big( g^{-1}A\big)\,.
$$
One checks that $\Volumizer$ is a Riemannian submersion in the normalization 
of \eqref{volumemetric}.
Furthermore, for any $g$ in $\Met(M)$ there is a map 
$\iota_g\colon \mathrm{Vol}(M) \times \{g\} \to \MetM$ given by 
$$
\iota_g(\nu) = \Big(\frac{d\nu}{d\mu_g}\Big)^{2/n} g \,,
$$
which is an isometric embedding. 

For any $\mu\in \mathrm{Vol}(M)$ the inverse image $\MetmuM=\Xi^{-1}[\mu]$ 
can be given a structure of a submanifold in the space of Riemannian metrics 
whose volume form is $\mu$. 
Its tangent space at $g$ consists of symmetric bilinear forms $A$ with $\Tr_g(A)=0$. 
The metric \eqref{metricofmetrics} induces a metric 
on $\MetmuM$, which turns it into a globally symmetric space. 

The natural action on $\MetmuM$ is again given by pull-back by elements of 
the group $\Diffmu(M)$. 
The triple of $\MetM$ (total space) together with $\MetmuM$ (fibre) 
and $\mathrm{Vol}(M)$ (base) forms a trivial fiber bundle, 
even though \eqref{metricofmetrics} is not a product metric.

Sectional curvature of 
the metric \eqref{metricofmetrics} on $\Met(M)$ was computed in \cite{fg} 
and found to be nonpositive. 
The corresponding sectional curvature of $\MetmuM$ is also nonpositive. 
On the other hand, the space $\mathrm{Vol}(M)$ 
equipped with $L^2$-metric \eqref{volumemetric} turns out to be flat. 

We now explain how these structures relate to our paper. 

%%%%%%%%%%%%%%%%%%%%%%%%%%%
\subsection{Induced $H^1$ metrics on orbits of the diffeomorphism group} 
\label{subsec:induced} 

Observe that the pull-back actions of $\Diff(M)$ on $\Met(M)$ and $\mathrm{Vol}(M)$ 
(and similarly, the action of $\Diffmu(M)$ on $\MetmuM$) 
leave the corresponding metrics \eqref{metricofmetrics} and \eqref{volumemetric} invariant. 
This allows one to construct geometrically natural right-invariant metrics 
on the orbits of a (suitably chosen) metric or volume form. 

We first consider the action of the full diffeomorphism group $\DiffM$ 
on the space of Riemannian metrics $\MetM$. 

\begin{proposition}\label{fulldiffeoorbit}
If $g\in\Met(M)$ has no nontrivial isometries then the map
$\pullbackg\colon \Diff(M)\to \Met(M)$
is an immersion and the metric \eqref{metricofmetrics} 
induces a right-invariant metric on $\Diff(M)$ given at the identity by 
\begin{equation}\label{deformationmetric}
\llangle u, v \rrangle 
= 
\llangle \Lie_ug, \Lie_vg\rrangle_g
\end{equation}
for any $u,v\in T_e \Diff(M)$.
\end{proposition}
\begin{proof}
First, observe that the differential of the pull-back map $\pullbackg(\eta)$ 
with respect to $\eta$ is given by the formula 
$$ 
\mathcal{P}_{g\ast\eta}(v\circ \eta) = \eta^*(\Lie_vg), 
$$
for any $v\in T_e \Diff(M)$ and $\eta\in\Diff(M)$, 
where $\Lie_v$ stands for the Lie derivative. 
If $g$ has no nontrivial isometries then it has no Killing fields 
and therefore the differential $\mathcal{P}_{g\ast}$ is a one-to-one map.

It follows that the induced Riemannian metric on the orbit through $g$ 
can be computed at any $\eta \in \DiffM$ as 
\begin{align*} 
\llangle u\circ\eta, v\circ\eta\rrangle_{\eta} 
&= 
\llangle \mathcal{P}_{g\ast}(u\circ\eta), \mathcal{P}_{g\ast}(v\circ\eta) \rrangle_{\eta^*g}  \\ 
&= 
\llangle \eta^*\Lie_ug, \eta^*\Lie_vg \rrangle_{\eta^*g}
\end{align*} 
for any $u,v\in T_e\Diff(M)$. 
The induced metric is necessarily right-invariant. 
In fact, writing out explicitly the pull-back in the form 
$\eta^\ast h = D\eta \, (h\circ\eta) D\eta^\transpose$ 
for any $h$ in $\Met(M)$ and using the definition of $\Tr_g$ and the change of variables formula 
one readily checks that 
\begin{align*} 
\llangle \eta^*\Lie_ug, \eta^*\Lie_vg\rrangle_{\eta^*g} 
&= 
\int_{M} \Tr_{\eta^\ast g} \big( \eta^\ast \Lie_u g \, \eta^\ast \Lie_v g \big) \, d\mu_{\eta^\ast g} 
\\ 
&= 
\int_M \Tr 
\Big(  (\eta^\ast g)^{-1} \eta^\ast \Lie_u g \, (\eta^\ast g)^{-1} \eta^\ast \Lie_v g  \Big) 
\sqrt{ \det(\eta^\ast g)} \, dx   \\ 
&= 
\int_M \Tr_g \big( \Lie_u g \, \Lie_v g \big) d\mu_g   \\ 
&=
\llangle \Lie_ug, \Lie_vg\rrangle_g  
\end{align*} 
holds for any $\eta\in\Diff(M)$, which implies right-invariance. 
\end{proof}

\begin{remark} \upshape 
More generally, if $g$ has non-trivial isometries, then the above procedure 
yields a right-invariant metric on the homogeneous space $\Diff(M)/\IsogM$; 
see the diagram \eqref{bigpicturediagram} below. 
\end{remark} 

\begin{remark} \upshape 
The inner product \eqref{deformationmetric} can be rewritten in terms of 
the exterior differential $d$ and its adjoint $\delta$ so that at the identity we get 
\begin{equation}\label{hodgemetricform}
\begin{split}
\llangle u, v\rrangle 
&=
\int_M \langle \Lie_ug, \Lie_vg\rangle\, d\mu \\
&= 
2\int_M \langle du^{\flat}, dv^{\flat}\rangle \, d\mu 
+ 
4\int_M \langle \delta u^{\flat}, \delta v^{\flat}\rangle \, d\mu 
- 
4\int_M \Ric(u,v) \, d\mu,
\end{split}
\end{equation}
for any vector fields $u, v \in T_e\Diff(M)$ and where $\Ric$ stands for 
the Ricci curvature of $M$. 
If the metric $g$ is Einstein then 
$\Ric(u,v) = \lambda \langle u,v\rangle$ 
for some constant $\lambda$ 
and the induced metric in \eqref{deformationmetric} 
becomes a special case of the $a$-$b$-$c$ metric \eqref{abcmetric} 
with $a=-4\lambda$, $b=4$ and $c=2$.
\end{remark} 

In exactly the same manner we obtain an immersion of 
the volumorphism group $\Diffmu(M)$ into $\Metmu(M)$.

\begin{proposition}\label{volumorphismorbit}
If $g\in\MetmuM$ has no nontrivial isometries then the map
$\pullbackg\colon \Diffmu(M)\to \MetmuM$
is an immersion and \eqref{metricofmetrics} restricts to 
a right-invariant metric on $\Diffmu(M)$. 
\end{proposition}

In this case \eqref{hodgemetricform} yields the formula
\begin{equation}\label{hodgevolume}
\llangle u,v\rrangle 
= 
2\int_M \langle du^{\flat}, dv^{\flat}\rangle \, d\mu - 4\int_M \Ric(u,v)\, d\mu,
\end{equation}
for any $u, v \in T_e \Diffmu(M)$.

\medskip

Finally, we perform an analogous construction for the action of $\Diff(M)$ 
on the space of volume forms $\mathrm{Vol}(M)$. 
In this case the isotropy subgroup is $\Diffmu(M)$ and 
we obtain a metric on the quotient space $\Diff(M)/\Diffmu(M)$. 

\begin{proposition}\label{volumeorbit}
If $\mu$ is a volume form on $M$ then the map 
$\pullbackmu\colon \Diff(M)\to \mathrm{Vol}(M)$ 
defines an immersion of the homogeneous space 
$\VolM$ into $\mathrm{Vol}(M)$ 
and the right-invariant metric induced by \eqref{volumemetric} 
has the form 
\begin{equation}\label{divdiv}
\llangle u, v\rrangle 
=
\llangle \Lie_u\mu, \Lie_v\mu \rrangle_\mu  
= 
\frac{4}{n} 
\int_M \diver{u} \cdot \diver{v} \, d\mu.
\end{equation}
\end{proposition}
\begin{proof} 
The differential of the pullback map is 
$$
\mathcal{P}_{\mu\ast\eta}(v\circ\eta) = \eta^\ast (\Lie_v \mu)
$$
for any $v \in T_e \DiffM$ and $\eta \in \DiffM$. 
Using formula \eqref{volumemetric} and changing variables as before we obtain 
\begin{align*}
\llangle u\circ\eta, v\circ\eta \rrangle_{[\eta]} 
&= 
\llangle \eta^\ast \Lie_u\mu, \eta^\ast \Lie_v\mu \rrangle_{\eta^\ast\mu}     \\ 
&= 
\frac{4}{n} \int_M (\diver{u}\circ\eta) \cdot (\diver{v}\circ\eta) \, d(\eta^\ast \mu)   \\ 
&=
\frac{4}{n} \int_M \diver{u} \cdot \diver{v} \, d\mu
= 
\llangle u, v \rrangle \,, 
\end{align*} 
so that the induced metric is again right-invariant. 
\end{proof} 

The three immersions described in Propositions \ref{fulldiffeoorbit}, \ref{volumorphismorbit}, and \ref{volumeorbit} 
can be summarized in the following diagram. 
\medskip
\begin{equation}\label{bigpicturediagram}
\begin{CD}
\IsogM @>{\imbedding}>> \DiffmuM @>{\projection}>> \DiffmuM/\IsogM
@>\pullbackg>>\MetmuM \\
@|     @VV{\imbedding}V    @VV{\imbedding}V    @VV{\imbedding}V\\
\IsogM @>{\imbedding}>> \DiffM @>{\projection}>> \DiffM/\IsogM
@>\pullbackg>>
\MetM \\
@VV{\imbedding}V    @|   @VV{\projection}V   @VV{\Volumizer}V \\
\DiffmuM @>{\imbedding}>> \DiffM @>{\projection}>> \DiffM/\DiffmuM
@>\pullbackmu>> \mathrm{Vol}(M)
\end{CD}
\end{equation}
\medskip

The first three terms of each row in \eqref{bigpicturediagram} form smooth fiber bundles 
in the obvious way. 
The third column is a smooth fiber bundle since $\IsogM \subset \DiffmuM$. 
The fourth column is a trivial fiber bundle which already appeared in \cite{fg}.

\begin{remark}  \upshape 
While curvatures of the spaces $\MetM$, $\MetmuM$ and $\mathrm{Vol}(M)$ 
have relatively simple expressions,
the induced metrics above on the corresponding homogeneous spaces 
$$
\Diff(M)/\IsogM, 
\quad 
\Diffmu(M)/\IsogM 
\quad \text{and}\quad 
\Diff(M)/\Diffmu(M) 
$$ 
turn out to have complicated geometries with one notable exception: 
the immersion of $\Diff(M)/\Diffmu(M)$ into $\mathrm{Vol}(M)$. 
This  immersion has codimension one 
and, as we discussed throughout this paper,
 $\pullbackmu$ maps the whole of $\DiffM$ 
into the space $\VolM\subset \mathrm{Vol}(M)$  of volume forms which integrate over $M$
to the same constant volume as $\mu$. 
On the other hand, the second fundamental form of the immersion of 
the quotient space $\Diff(M)/\IsogM$ into $\Met(M)$ 
is complicated (in this case codimension is infinite).
One can show that 
the sectional curvature of $\Diff(M)/\IsogM$ in the induced metric 
assumes both signs, 
see \cite{KLMP}. 
\end{remark}

%%%%%%%%%%%%%%%%%%%%%%%%%%%%%%%%%%%%%
%%%%%%%%%%%%%%%%%%%%%%%%%%%%%%%%%%%%%
%%%%%%%%%%%%%%%%%%%%%%%%%%%%%%%%%%%%%

%%%%%%%%%%%%%%%%%%%%%%%%%%%%%%%%%%%%%%%%%%%%%%%%%%%%%%%%%%%%%%%%%%%
%%%%%%%%%%%%%%%%%%%%%%%%%%%%%%%%%%%%%%%%
\bibliographystyle{plain}

%\bibliography{is}

\begin{thebibliography}{99999o}
\small

%\bibitem[Am]{amari}
%S. Amari,
%\emph{Differential-geometrical Methods in Statistics}, Springer-Verlag, Berlin 1985.

\bibitem{AN}
S. Amari and H. Nagaoka, 
\emph{Methods of Information Geometry}, American Mathematical Society, Providence, RI 2000. 

\bibitem{ags} 
L. Ambrosio, N. Gigli and G. Savare, 
\textit{Gradient flows in metric spaces and in the space of probability measures}, 
Birkhauser, Basel 2005. 

\bibitem{A}
V. Arnold, 
Sur la g\'eometrie diff\'erentielle des groupes de Lie de dimension infinie et 
ses application \`a l'hydrodynamique des fluides parfaits, 
{\it Ann. Inst. Fourier (Grenoble)} {\bf 16} (1966), 319--361.

\bibitem{ak} 
V. Arnold and B. Khesin, 
\textit{Topological Methods in Hydrodynamics}, 
Springer, New York 1998. 

%\bibitem[BBT]{BBT}
%O.~Babelon, D.~Bernard, and M.~Talon, 
%\emph{Introduction to classical integrable systems}, Cambridge University Press, 
%Cambridge, UK 2003.

\bibitem{bb} 
J.-D. Benamou and Y. Brenier, 
A computational fluid mechanics solution to the Monge-Kantorovich 
mass transfer problem, 
\textit{Numer. Math.} \textbf{84} (2001), 375--393.

%\bibitem[BE]{BE}
%M. Berger and D. Ebin, 
%Some decompositions of the space of symmetric tensors on a Riemannian manifold,
%\textit{J. Diff. Geom.} \textbf{3} (1969) 379--392.

\bibitem{be} 
A. Besse, 
\textit{Einstein Manifolds}, 
Springer, New York 1987. 

\bibitem{Bolsinov}
A. Bolsinov, Integrable geodesic flow on homogeneous spaces,
Appendix C in {\it Modern methods in the theory of integrable systems} 
by A.V.~Borisov, I.S.~Mamaev,  Izhevsk, 2003, 236--254, 
and personal communication (2010).

\bibitem{bc} 
A. Bressan and A. Constantin, 
Global solutions of the Hunter-Saxton equation, 
\textit{SIAM J. Math. Anal.} \textbf{37} (2005), 996--1026. 

%\bibitem[CE]{cheegerebin}
%J. Cheeger and D. Ebin,
%\emph{Comparison theorems in Riemannian geometry},
%North Holland, Amsterdam, 1975.

\bibitem{chentsov}
N.N. Chentsov, 
\emph{Statistical decision rules and optimal inference}, American Mathematical
Society, Providence, RI 1982.

%\bibitem[C2]{chentsov2} 
%N.N. Chentsov, 
%The unfathomable influence of Kolmogorov, 
%\textit{Ann. Statist.} \textbf{18} (1990), 987--998. 

\bibitem{CDK}
G. Cupini, B. Dacorogna, and O. Kneuss,
On the equation $\det\nabla u=f$ with no sign hypothesis,
{\it Calc. Var. Partial Differential Equations} {\bf 36} (2009), no. 2, 251--283.

%\bibitem[D]{dawid}
%A.P. Dawid, 
%Further comments on some comments on a paper by Bradley Efron,
%{\it Ann. Statist.} {\bf 5} (1977) no. 6, 1249. 

\bibitem{ebin}
D. Ebin, 
The manifold of Riemannian metrics, 
in \textit{Proc. Sympos. Pure Math.} {\bf 15}, 
Amer. Math. Soc., Providence, R.I., 1970.

\bibitem{Ebin-Marsden}
D. Ebin and J.E. Marsden, 
Groups of diffeomorphisms and the motion of an incompressible fluid, 
{\it Ann. Math.} {\bf 92} (1970), 102--163.

\bibitem{efron}
B. Efron, Defining the curvature of a statistical problem
(with applications to second order efficiency), 
{\it Ann. Statist.} {\bf 3} (1975) no. 6, 1189--1242.

\bibitem{F2006}
A.S. Fokas, Integrable nonlinear evolution PDEs in $4+2$ and $3+1$ dimensions, {\it Phys. Rev. Lett.} {\bf 96} (2006), 190201. 

%\bibitem[F2]{F2008}
%A. S. Fokas, Soliton multidimensional equations and integrable evolutions preserving Laplace�s equa- 
%tion, {\it Physics Letters A} {\bf 372} (2008), 1277-1279. 

%\bibitem[F3]{F20082}
%A. S. Fokas, The D-Bar method, inversion of certain integrals and integrability in $4 + 2$ and $3 + 1$ 
%dimensions, {\it J. Phys. A} {\bf 41} (2008), 1--16. 

\bibitem{fg}
D.S. Freed and D. Groisser, 
The basic geometry of the manifold of Riemannian metrics and of its quotient by 
the diffeomorphism group, 
{\it Michigan Math. J.} {\bf 36} (1989) 323--344.

%\bibitem[GS]{GS}
%A.L. Gibbs, F.E. Su, 
%On choosing and bounding probability metrics,
%{\it Int. Stat. Rev.} {\bf 70} no. 3 (2002), 419--435.

\bibitem{GE} 
M.L. Gromov and Y. Eliashberg,
Construction of a smooth mapping with a prescribed Jacobian. I,
{\it Functional Anal. Appl.} {\bf 7} (1973), 27--33.

\bibitem{HMA}
A.N. Hirani, J.E. Marsden, and J. Arvo,
Averaged template matching equations,
{\it Lect. Notes Comp. Sci.} {\bf 2134} (2001), 528--543.

%\bibitem[HM]{HM}
%D.D. Holm and J.E. Marsden, 
%Momentum maps and measure-valued solutions 
%(peakons, filaments, and sheets) for the EPDiff equation, 
%in \emph{The breadth of symplectic and Poisson geometry}, 
%Birkh\"auser, Boston 2004.

\bibitem{HMR}
D. Holm, J.E. Marsden, and T.S. Ratiu, 
The Euler-Poincar\'e equations and semidirect products with applications 
to continuum theories, 
{\it Adv. Math.} {\bf 137} (1998), 1--81.

%\bibitem[HL]{HL}
%T. Hou and C. Li,
%On global well-posedness of the Lagrangian averaged Euler equations,
%{\it SIAM J. Math. Anal.} {\bf 38} (2006), 782--794.

\bibitem{H-S}
J.K. Hunter and R. Saxton, 
Dynamics of director fields,
{\it SIAM J. Appl. Math.} {\bf 51} (1991), 1498--1521.

%\bibitem[HZ]{H-Z}
%J.K. Hunter and Y. Zheng, 
%On a completely integrable nonlinear hyperbolic variational equation, 
%{\it Physica D} {\bf 79} no. 2�4 (1994), 361�386.


%\bibitem[JS]{jacshir} 
%J. Jacod and A.N. Shiryaev, 
%\textit{Limit Theorems for Stochastic Processes}, 
%Springer-Verlag, Berlin 1987.

%\bibitem[Ka]{kakutani} 
%S. Kakutani, 
%On equivalence of infinite product measures, 
%\textit{Ann. of Math.} \textbf{49} (1948), 214--224. 

%\bibitem[KLM]{K-L-M}
%B. Khesin, J. Lenells, and G. Misio\l ek, 
%Generalized Hunter-Saxton equation and the geometry of the group of circle diffeomorphisms, 
%{\it Math. Ann.} {\bf 342} (2008), 617--656. 

\bibitem{K-M}
B. Khesin and G. Misio\l ek, 
Euler equations on homogeneous spaces and Virasoro orbits, 
{\it Adv. Math.} {\bf 176} (2003), 116--144.
       
%\bibitem[KM2]{K-M2}
%B. Khesin and G. Misio{\l}ek, 
%Shock waves for the Burgers equation and 
%curvatures of diffeomorphism groups,
%{\it Proc. Steklov Math. Inst.} {\bf 259} (2007), 73--81.

\bibitem{KLMP} 
B. Khesin, J. Lenells, G. Misio\l ek and S.C. Preston, 
Curvatures of Sobolev metrics on diffeomorphism groups, 
\textit{in preparation}. 

\bibitem{ky} 
A. Kirillov and D. Yurev, 
K\"{a}hler geometry of the infinite-dimensional homogeneous space 
$M = \Diff^+ (S^1)/\mathrm{Rot}(S^1)$, 
\textit{Funct. Anal. App.} \textbf{21} (1987). 
 
%\bibitem[K]{kolev}
%B. Kolev, 
%Geometric differences between the Burgers and the Camassa-Holm equation, 
%{\it J. Nonlinear Math. Phys.} {\bf 15} Suppl. 2 (2008), 116--132.
       
%\bibitem[La]{la} 
%S. Lang, 
%\textit{Fundamentals of Differential Geometry}, 
%Springer-Verlag 1999. 

%\bibitem[LC]{lecam} 
%L.M. LeCam, 
%\textit{Asymptotic Methods in Statistical Decision Theory}, 
%Springer-Verlag, New York 1986. 

\bibitem{L1}
J. Lenells, 
The Hunter-Saxton equation describes the geodesic flow on a sphere, 
{\it J. Geom. Phys.} {\bf 57} (2007), 2049--2064.

\bibitem{L2}
J. Lenells, 
The Hunter-Saxton equation: a geometric approach,
{\it SIAM J. Math. Anal.} {\bf 40} (2008), 266--277.

\bibitem{L3}
J. Lenells,
Weak geodesic flow and global solutions of the Hunter-Saxton equation,
{\it Discrete Contin. Dyn. Syst.} {\bf 18} (2007), no. 4, 643--656.

\bibitem{lmt} 
J. Lenells, G. Misio\l ek and F. T\i\u{g}lay, 
Integrable evolution equations on spaces of tensor densities and their peakon solutions, 
\textit{Comm. Math. Phys.} \textbf{299} (2010), 129--161. 

%\bibitem[LV]{lv} 
%J. Lott and C. Villani, 
%Ricci curvature for metric-measure spaces via optimal transport, 
%\textit{Ann. Math} \textbf{169} (2009), 903--991. 

%\bibitem[Lu]{lukatsky}
%A.M. Lukatsky, 
%On the curvature of the diffeomorphisms group, 
%{\it Ann. Global Anal. Geom.} {\bf 11} (1993),  no. 2, 135--140. 

\bibitem{Matveev-Topalov} 
V.S. Matveev and P. Topalov,
Geodesic equivalence of metrics as a particular case of the integrability of 
geodesic flows. {\it Theoret. and Math. Phys.} {\bf  123} (2000), no. 2, 651--658. 

%\bibitem{mm} 
%P. Michor and D. Mumford, 
%An overview of the Riemannian metrics on spaces of curves 
%using the Hamiltonian approach, 
%\textit{Appl. Comp. Harmonic Anal.} \textbf{23} (2007), 74--113. 

%\bibitem[Mi1]{MStab}
%G. Misiolek, 
%Stability of flows of ideal fluids and the geometry of the group of diffeomorphisms, 
%{\it Indiana Univ. Math. J.} {\bf 42} (1993), 215--235.

%\bibitem[Mi1]{MKdV}
%G. Misio\l ek, 
%Conjugate points in the Bott-Virasoro group and the KdV equation, 
%{\it Proc. Amer. Math. Soc.} {\bf 125} (1997), 935--940

%\bibitem[Mi]{M}
%G. Misio\l ek, 
%A shallow water equation as a geodesic flow on the Bott-Virasoro group, 
%{\it J. Geom. Phys.} {\bf 24} (1998), 203--208.

\bibitem{MP}
G. Misio\l ek and S.C.~Preston, 
Fredholm properties of Riemannian exponential maps on diffeomorphism groups, 
{\it Invent. math.} {\bf 179} no. 1., 191--227.

\bibitem{moser}
J. Moser, 
On the volume elements on a manifold, 
{\it Trans. Amer.Math. Soc.} {\bf 120} (1965), 286--294. 

\bibitem{otto} 
F. Otto, 
The geometry of dissipative evolution equations: the porous medium equation, 
\textit{Comm. Partial Differ. Eq.} \textbf{26} (2001), 101--174. 

%\bibitem[OK]{O-K}
%V. Ovsienko and B. Khesin, The (super) KdV equation as an Euler equation,
%{\it Funct. Anal. Appl.} {\bf 21:4} (1987), 81--82.

%\bibitem[Pe]{Pe}
%P. Petersen, {\it Riemannian geometry}, 
%Springer-Verlag, New York, 2006.
% xvi+401 pp.

%\bibitem[Po]{pol} 
%D. Pollard, 
%\textit{A User's Guide to Measure Theoretic Probability}, 
%Cambridge University Press, 2002.

%\bibitem[Pr]{PNonpositive}
%S.C. Preston, 
%Nonpositive curvature on the area-preserving diffeomorphism group, 
%{\it J. Geom. Phys.} {\bf 53} (2005), 226--248.

\bibitem{rao}
C.R. Rao, 
Information and the accuracy attainable in the estimation of statistical parameters,
reprinted in {\it Breakthroughs in statistics: foundations and basic theory}, 
S. Kotz and N. L. Johnson, eds., Springer, New York, 1993. 

%\bibitem[Ro]{rouchon}
%P. Rouchon,
%The Jacobi equation, Riemannian curvature and the motion of a perfect incompressible fluid,
%{\it European J. Mech. B Fluids} {\bf 11} (1992).

%\bibitem[Shk]{shkoller}
%S. Shkoller, Geometry and curvature of diffeomorphism groups with $H^1$ metric and 
%mean hydrodynamics, 
%{\it J. Funct. Anal.} \textbf{160} (1998).

\bibitem{mm} 
E. Sharon and D. Mumford, 
2D-Shape Analysis Using Conformal Mapping
\textit{Int. J. of Comp. Vision} \textbf{70}(1) (2006), 55--75.

\bibitem{shn} 
A. Shnirelman, 
Generalized fluid flows, their approximation and applications, 
\textit{Geom. funct. anal.} \textbf{4} (1994). 

%\bibitem[Sm]{sm} 
%N.K. Smolentsev, 
%Diffeomorphism groups of compact manifolds, 
%\textit{J. Math. Sci.} \textbf{146} (2007). 

%\bibitem[SJJ]{SJJ}
%A. Srivastava, I. Jermyn, and S. Joshi, 
%Riemannian analysis of probability density functions with applications in vision,
%in Proc. IEEE Computer Vision and Pattern Recognition (CVPR), (2007), 1--8.

\bibitem{st} 
K-T. Sturm, 
On the geometry of metric measure spaces. I and II, 
\textit{Acta Math.} \textbf{196} (2006), 65--131, 133--177. 

\bibitem{Tabachnikov} 
S. Tabachnikov,
Projectively equivalent metrics, exact transverse line fields and the geodesic 
flow on the ellipsoid, 
\textit{Comment. Math. Helv.} \textbf{74} (1999), no. 2, 306--321. 

\bibitem{tt} 
L. Takhtajan and L. Teo, 
Weil-Petersson metric on the universal Teichm\"{u}ller space, 
\textit{Mem. Amer. Math. Soc.} \textbf{183} (2006), no. 861, viii+119 pp. 

\bibitem{taylor}
M.E. Taylor,
\textit{Finite and Infinite Dimensional Lie Groups and Evolution Equations}, 
lecture notes from Chapel Hill, 2003. 

%\bibitem[TV]{tv} 
%F. T\i\u{g}lay and C. Vizman, 
%Generalized Euler-Poincare equations on Lie groups and
%homogeneous spaces, orbit invariants and applications, 
%\textit{preprint} (2010). 

\bibitem{v} 
C. Villani, 
\textit{Optimal Transport: Old and New}, 
Springer, Berlin, 2009. 

%\bibitem{YMSM}
%L. Younes, P.W. Michor, J. Shah, and D.~Mumford, 
%A metric on shape space with explicit geodesics, 
%\emph{Rend. Lincei Mat. Appl.} \textbf{9} (2008), 25--57.

%\bibitem[Zo]{zolotarev} 
%V.M. Zolotarev, 
%On analogs of the Hellinger metric,
%\textit{J. Math. Sci.} \textbf{59} (1992),  no. 4, 921--925. 

\end{thebibliography}

\end{document}